\RequirePackage{amsthm}
\documentclass[sn-mathphys-num,pdflatex]{sn-jnl}

\usepackage{graphicx}%
\usepackage{multirow}%
\usepackage{amsmath,amssymb,amsfonts}%
\usepackage{mathrsfs}%
\usepackage{algorithm}%
\usepackage{algorithmicx}%
\usepackage{algpseudocode}%

\newtheorem{theorem}{Theorem}%
\newtheorem{lemma}[theorem]{Lemma}%
\newtheorem{example}{Example}%
\newtheorem{remark}{Remark}%
\newcounter{assumcounter}
\newtheorem{assum}[assumcounter]{Assumption}%
\newtheorem{definition}{Definition}%

\usepackage{array} 
\usepackage{enumerate} 
\usepackage{enumitem}
\usepackage{hyperref}
\usepackage{lmodern}

\newcommand{\eps}{\varepsilon}
\newcommand{\R}{\mathbb{R}}
\newcommand{\N}{\mathbb{N}}
\newcommand{\calH}{\mathcal{H}}
\newcommand{\J}{\mathcal{J}}
\newcommand{\Jeps}{\mathcal{J}_\varepsilon^2 f}
\newcommand{\Txeps}{\mathcal{T}_{x,\varepsilon}}
\newcommand{\Rxeps}{R_{x,\varepsilon}}
\newcommand{\TWxeps}{\mathcal{T}^W_{x,\varepsilon}}
\newcommand{\T}{\mathcal{T}}

\newcommand{\deriv}{\mathrm{d}}

\DeclareMathOperator*{\conv}{conv}

\DeclareMathOperator*{\pr}{pr}
\DeclareMathOperator*{\argmin}{arg\,min}
\DeclareMathOperator*{\argmax}{arg\,max}

\newcommand{\sogs}{\texttt{SOGS}}
\newcommand{\gradsamp}{\texttt{Gradsamp}}
\newcommand{\dgs}{\texttt{DGS}}
\newcommand{\hanso}{\texttt{HANSO}}
\newcommand{\slqpgs}{\texttt{SLQPGS}}
\newcommand{\lmbm}{\texttt{LMBM}}
\newcommand{\superpolyak}{\texttt{SuperPolyak}}
\newcommand{\vubundle}{\texttt{VUbundle}}
\newcommand{\vu}{$\mathcal{VU}$}

\begin{document}

\title{Using second-order information in gradient sampling methods for nonsmooth optimization}

\author*[1]{\fnm{Bennet} \sur{Gebken}}\email{bennet.gebken@cit.tum.de}

\affil*[1]{\orgdiv{Department of Mathematics}, \orgname{Technical University of Munich}, \orgaddress{\street{Boltzmannstr. 3}, \city{Garching b. München}, \postcode{85748}, \country{Germany}}}

\abstract{In this article, we introduce a novel concept for second-order information of a nonsmooth function inspired by the Goldstein $\eps$-subdifferential. It comprises the coefficients of all existing second-order Taylor expansions in an $\eps$-ball around a given point. Based on this concept, we define a model of the objective as the maximum of these Taylor expansions, and derive a sampling scheme for its approximation in practice. Minimization of this model induces a simple descent method, for which we show convergence for the case where the objective is convex or of max-type. While we do not prove any rate of convergence of this method, numerical experiments suggest superlinear behavior with respect to the number of oracle calls of the objective.}


\maketitle

\section{Introduction} \label{sec:introduction}

    Nonsmooth optimization is concerned with optimizing functions that are not necessarily differentiable. Solution methods from smooth optimization, like gradient descent, may fail to work in this case, as even when the gradient $\nabla f(x)$ of a nonsmooth function $f$ at a point $x$ exists, it may not yield a stable description of the local behavior of $f$ around $x$. This issue can be overcome by considering multiple gradients at the same time: In \emph{bundle methods} \cite{MN1992,BKM2014}, gradients from previous iterations are collected to build a \emph{cutting-plane model} of the objective at the current point. Alternatively, in \emph{gradient sampling methods} \cite{BLO2005,BCL2020,MY2012,GP2021,ZLJ2020,DDL2022}, $\nabla f(x)$ is replaced by (an approximation of) the \emph{Goldstein} $\eps$\emph{-subdifferential} $\partial_\eps f(x)$ \cite{G1977}. For several algorithms from these classes, a linear speed of convergence (at best) was proven \cite{HSS2017,R1999,DG2023}. Since for smooth optimization, Newton's method yields (at least) superlinear convergence by using second-order derivatives, the question arises whether such a fast method can be constructed for the nonsmooth case as well. It turns out that this task is quite challenging: In \cite{MS2012} it was called a ``wondrous grail'' and it can be related Problem 14 in \cite{H2007}.
    
    To the best of our knowledge, there are currently two methods for which superlinear convergence can be proven for certain classes of nonsmooth optimization problems. The first method is the bundle-based \vu\emph{-algorithm} from \cite{MS2005}. The idea of this method is to identify a subspace $\mathcal{U}$ along which $f$ behaves smoothly and to then perform a Newton-like step tangent to $\mathcal{U}$. One of the challenges of this approach is to identify the subspace $\mathcal{U}$ automatically. It appears that the only mechanism that is guaranteed to achieve this is described in \cite{DSS2009}, and requires the objective to be piecewise twice continuously differentiable with convex selection functions (with some additional technical assumptions). The second method is the recent \emph{SuperPolyak} \cite{CD2024}. It is based on the subproblem of the bundle method by Polyak \cite{P1969}, but uses iteratively computed bundle information and treats the cutting-plane inequalities as equalities. The objective functions it is designed for are functions that possess a \emph{sharp} minimum (i.e., the objective value grows linearly away from the minimum) with known optimal value. While these are the only methods (that we are aware of) for which superlinear convergence can be proven, there are several other methods that attempt to achieve this rate by using some form of second-order information: In \cite{M1984,LV1998,G2002}, methods were proposed that arise by infusing second-order information into cutting-plane models. Alternatively, there are methods that directly generalize quasi-Newton to the nonsmooth case \cite{CO2012,CQ2015,CRZ2019,LV1999,H2004,LO2013,CMO2017}. However, there are certain drawbacks in all of the approaches mentioned this far: Both the \vu-algorithm and SuperPolyak require certain special structures of $f$. For \cite{M1984,LV1998,G2002}, a proper definition of ``second-order information'' of a nonsmooth function is missing, which makes it difficult to exploit this information theoretically for proving fast convergence. Similarly, for the generalized quasi-Newton approaches, there is no corresponding ``nonsmooth Newton's method'' from which their convergence theory can be inherited.

    The goal of this article is to construct a fast method that avoids the above drawbacks. The basic idea is to use the maximum of existing second-order Taylor expansions around a point as a local model of the objective function, as was already proposed in \cite{LV1998}. There, whenever the Hessian matrix does not exist at a point $y$, ``the Hessian at some point infinitely close to $y$'' (cf.\ \cite{LV1998}, p.\ 2) is taken instead. In contrast to this, we will base our model on the \emph{second-order} $\eps$\emph{-jet}, which can be seen as a second-order version of the Goldstein $\eps$-subdifferential. It contains the coefficients of all existing second-order Taylor expansions (and their limits) in an $\eps$-ball around a given point and allows us to define the aforementioned model in a well-defined way. Based on minimization of this model, we first derive a purely abstract descent method (Algo.\ \ref{algo:abstract_descent_method}) for the case where the entire $\eps$-jet is available at every point. To obtain a method that is actually implementable, we then derive a sampling strategy to approximate the $\eps$-jet (Algo.\ \ref{algo:approx_jet}), for which availability of a single element of the $0$-jet at every point (i.e., almost always simply the objective value, gradient and Hessian at that point)
    is sufficient. Inserting this sampling scheme into the abstract method yields a practical descent method (Algo.\ \ref{algo:practical_descent_method}), for which we prove global convergence for the case where the objective is convex or of max-type. We do not prove anything about the rate of convergence of the abstract or practical method here. However, numerical experiments suggest that it is fast with respect to oracle calls (i.e., evaluations of $f$ and its derivatives). For reproducibility and transparency, the code for all of our experiments is available at \url{https://github.com/b-gebken/SOGS}.

    It is important to note that fast convergence \textit{with respect to oracle calls} does not necessarily imply a fast runtime. For example, when comparing gradient descent and Newton's method in smooth optimization, the latter yields superlinear convergence with respect to oracle calls. But every iteration of Newton's method requires the solution of a system of linear equations, which is more expensive than an iteration of gradient descent, and may outweigh the fast convergence (see, e.g., \cite{GBC2016}, Section 8.6.1). In a similar vein, every iteration of our method will require the solution of a subproblem (see \eqref{eq:min_TxepsW_quad}) which is a nonconvex \emph{quadratically constrained quadratic program (QCQP)} (see, e.g., \cite{BV2004}, Section 4.4).
    We could apply several modifications to simplify its solution, like positive definite modifications of the (potentially nonconvex) quadratic terms (as in \cite{LV1998}), summing the quadratic terms and moving them to the objective of the subproblem (as in \cite{G2002}) or using a quasi-Newton matrix for the second-order information to avoid evaluation of the Hessian matrix. However, we refrain from doing so, as we believe that it is important to first understand whether superlinear convergence can be obtained in the best possible setting (with high computational effort). If this is not the case, then the above techniques likely do not lead to superlinear convergence either.
    To put it in another way: We first want to figure out what Newton's method for nonsmooth optimization could be in this article, before trying to derive quasi-Newton variants of it in future work.

    The remainder of this article is structured as follows: In Section \ref{sec:preliminaries} we introduce the basics of nonsmooth analysis. In Section \ref{subsec:second_order_eps_jet} we define the second-order $\eps$-jet and derive some of its theoretical properties. Afterwards, in Section \ref{subsec:second_order_model}, we define a model based on the $\eps$-jet and prove error estimates for the case where $f$ is convex or of max-type. In Section \ref{sec:descent_method} we derive our descent method. We begin with the abstract version in Section \ref{subsec:abstract_algorithm} and prove its convergence. We then derive the sampling scheme for the $\eps$-jet and prove its termination in Section \ref{subsec:approximating_eps_jet}. In Section \ref{subsec:practical_algorithm}, we insert the sampling scheme into the abstract method to obtain the practical descent method, for which we prove convergence as well. Section \ref{sec:numerical_experiments} contains our numerical experiments. We first compare the performance of our method to other solvers for nonsmooth optimization problems using common test problems. Afterwards, we compare it to the \vu-algorithm and SuperPolyak, which both have superlinear speed of convergence on their respective problem classes. Finally, in Section \ref{sec:conclusion}, we summarize our results and discuss possible directions for future research.

\section{Preliminaries} \label{sec:preliminaries}

    In this section, we introduce the basics of nonsmooth analysis that are used throughout the article. More thorough introductions can be found in \cite{C1990,BKM2014}. To this end, let $f : \R^n \rightarrow \R$ be locally Lipschitz continuous and let $\Omega$ be the set of points in which $f$ is not differentiable. By Rademacher's Theorem, $\Omega$ is a null set. The \emph{Clarke subdifferential of} $f$ \emph{at} $x \in \R^n$ is defined as
    \begin{equation} \label{eq:def_clarke_subdifferential}
        \begin{aligned}
            \partial f(x) := \conv &\left( \left\{ \xi \in \R^n : \exists (x^j)_j \in \R^n \setminus \Omega \text{ with } x^j \rightarrow x \text{ and } \nabla f(x^j) \rightarrow \xi \right\} \right),
        \end{aligned}
    \end{equation}
    and its elements are called \emph{subgradients}. It is nonempty and compact. A necessary condition for a point $x$ to be locally optimal is that $0 \in \partial f(x)$, which is referred to as $x$ being a \emph{critical} point. As a set-valued map, the map $x \mapsto \partial f(x)$ is upper semicontinuous. A more ``stable'' subdifferential which circumvents some of the practical issues of the Clarke subdifferential (cf.\ \cite{L1989}, Section 3) is the \emph{Goldstein} $\eps$\emph{-subdifferential} \cite{G1977}. For $x \in \R^n$ and $\eps \geq 0$ it can be defined as
    \begin{align} \label{eq:def_eps_subdifferential}
        \partial_\eps f(x) := \conv \left( \bigcap_{\delta > \eps} \overline{\{ \nabla f(y) : y \in B_\delta(x) \setminus \Omega \}} \right),
    \end{align}
    where $B_\delta(x) := \{ y \in \R^n : \| x - y \| \leq \delta \}$ and $\| \cdot \|$ is the Euclidean norm. Both subdifferentials are related via
    \begin{align} \label{eq:eps_subdiff_via_clarke}
        \partial_\eps f(x) = \conv(\partial f(B_\eps(x)))
    \end{align}
    (cf.\ \cite{K2007}, Eq.\ (2.1)).
    Like the Clarke subdifferential, $\partial_\eps f(x)$ is nonempty and compact. Furthermore, it can be used to compute descent directions of $f$: Let $\bar{v} = -\argmin_{\xi \in \partial_\eps f(x)} \| \xi \|$. Then the mean-value theorem (\cite{C1990}, Theorem 2.3.7) combined with a well-known result from convex analysis (cf.\ \cite{CG1959}) yields
    \begin{align} \label{eq:eps_subdiff_decrease}
        f(x + t \bar{v}) - f(x) \leq - t \| \bar{v} \|^2 
        \quad \forall t \in (0,\eps/\| \bar{v} \|].
    \end{align}
    For $k \in \N$ we say that a function is $C^k$ if it is $k$-times continuously differentiable.

\section{Second-order \texorpdfstring{\boldmath{$\eps$}}{epsilon}-jet and model} \label{sec:second_order_eps_jet_and_model}

In this section, we first define the second-order $\eps$-jet and derive some of its properties. Afterwards, we use it to define and analyze the second-order model that will later be the foundation of our descent method. This includes error estimates for the cases where $f$ is convex or of max-type.

This article contains results for $f$ belonging to different classes of functions. To avoid confusion, each result explicitly states which assumption is required.

\subsection{Second-order \texorpdfstring{\boldmath{$\eps$}}{epsilon}-jet} \label{subsec:second_order_eps_jet}

    We begin by introducing the class of functions to which our approach is applicable. To this end, let $f : \R^n \rightarrow \R$ and let $\Omega^2 \subseteq \R^n$ be the set of points in which $f$ is not twice differentiable.

    \begin{assum} \label{assum:A1}
        Assume that
        \begin{enumerate}[itemindent=25pt,label=(A1.\arabic*)]
            \item \label{enum:A1_1} $f$ is locally Lipschitz continuous,
            \item \label{enum:A1_2} $\Omega^2$ is a null set and
            \item \label{enum:A1_3} for all $x \in \R^n$ there is an open set $U \subseteq \R^n$ with $x \in U$ such that $\{ \nabla^2 f(y) : y \in U \setminus \Omega^2 \}$
            is bounded.
        \end{enumerate}
    \end{assum}

    First of all, \ref{enum:A1_1} allows us to use the classical theory of nonsmooth analysis from Section \ref{sec:preliminaries}. Since we want to define our second-order object by evaluating the Hessian matrix of $f$ at different points in $\R^n$, $\R^n \setminus \Omega^2$ should at least be dense. The stronger assumption \ref{enum:A1_2} is made to obtain consistency with the $\eps$-subdifferential later on (cf.\ Lemma \ref{lem:jet_projections}). Finally, \ref{enum:A1_3} assures boundedness. The second-order object that we use throughout the article is now defined as follows:

    \begin{definition} \label{def:second_order_eps_jet}
        Let $x \in \R^n$ and $\eps \geq 0$. For $f$ satisfying Assumption \ref{assum:A1}, the set
        \begin{align*}
            \Jeps(x) := \bigcap_{\delta > \eps} \overline{\{ (y,f(y),\nabla f(y), \nabla^2 f(y)) : y \in B_\delta(x) \setminus \Omega^2 \}} 
            \subseteq \R^n \times \R \times \R^n \times \R^{n \times n}
        \end{align*}
        is called the \emph{second-order $\eps$-jet of} $f$ \emph{at} $x$.\footnote{The term ``jet'' is inspired by the classical notion of jets for smooth functions (\cite{GG1973}, §2) and the related concept of \emph{Fréchet second order subjets} in \cite{IP1997}.}
    \end{definition}

    Compared to the definition of the Goldstein $\eps$-subdifferential (cf.\ \eqref{eq:def_eps_subdifferential}) there are two modifications: Firstly, since we later want to define the model based on second-order Taylor expansions, the gradient $\nabla f(y)$ is replaced by everything that is needed for building these expansions, represented as the $4$-tuple $(y,f(y),\nabla f(y), \nabla^2 f(y))$, and $\Omega$ is replaced by $\Omega^2$. Secondly, we omitted taking the convex hull because a convex combination of such $4$-tuples does in general not correspond to any Taylor expansion of $f$. In the following, we analyze some of the properties of $\Jeps(x)$. First of all, like the $\eps$-subdifferential, $\Jeps(x)$ is nonempty and compact:
    \begin{lemma} \label{lem:jet_properties}
        Assume that $f$ satisfies Assumption \ref{assum:A1}. Then $\Jeps(x)$ is nonempty and compact for all $x \in \R^n$ and $\eps \geq 0$.
    \end{lemma}
    \vspace{-15pt}
    \begin{proof}
        We first show that the set
        \begin{align} \label{eq:def_T_delta}
            T(\delta) := \{ (y,f(y),\nabla f(y), \nabla^2 f(y)) : y \in B_\delta(x) \setminus \Omega^2 \}
        \end{align}
        is bounded for all $\delta > \eps$. Since we consider boundedness in a finite-dimensional product space, it is equivalent to boundedness of the projections on the individual factors. Clearly, $B_\delta(x) \setminus \Omega^2$ and $f(B_\delta(x) \setminus \Omega^2)$ are bounded (due to continuity of $f$). Furthermore, $\{ \nabla f(y) : y \in B_\delta(x) \setminus \Omega^2 \}$ is bounded as a subset of the compact set $\partial_\delta f(x)$. To see that $\{ \nabla^2 f(y) : y \in B_\delta(x) \setminus \Omega^2\}$ is bounded as well, let $(U_i)_{i \in I}$ be an open cover of $B_\delta(x)$ induced by \ref{enum:A1_3} with corresponding upper bounds $(K_i)_{i \in I}$ (for any matrix norm). Due to compactness of $B_\delta(x)$ there is a finite subcover $(U_i)_{i \in I' \subseteq I}$. In particular, $\max_{i \in I'} K_i$ is an upper bound for $\{ \nabla^2 f(y) : y \in B_\delta(x) \setminus \Omega^2\}$. \\
        Thus, $T(\delta)$ is bounded and $\overline{T(\delta)}$ is compact. In particular, $\overline{T(\delta)}$ is nonempty due to density of $\R^n \setminus \Omega^2$ by \ref{enum:A1_2}. This means that $(\overline{T(\delta)})_{\delta > \eps}$ is a nested family of nonempty, compact sets and by Cantor's intersection theorem (see, e.g., Thm.\ 1 in \cite{H1985}), the intersection $\Jeps(x) = \bigcap_{\delta > \eps} \overline{T(\delta)}$ is nonempty and compact.
    \end{proof}

    The first-order information that is contained in the second-order $\eps$-jet is consistent with the $\eps$-subdifferential. To see this, let $\pr_i$, $i \in \{1,\dots,4\}$, be the projections of $\R^n \times \R \times \R^n \times \R^{n \times n}$ onto its factors. Then we have the following relationship:

    \begin{lemma} \label{lem:jet_projections}
        Assume that $f$ satisfies Assumption \ref{assum:A1}. Let $x \in \R^n$ and $\eps \geq 0$. Then
        \begin{align*}
            &\{ (y,\varphi,\xi)  : \exists \calH \in \R^{n \times n} \text{ with } (y,\varphi,\xi,\calH) \in \Jeps(x) \} \\
            &= \{ (y,f(y),\xi) : y \in B_\eps(x), \xi \in \partial f(y) \}.
        \end{align*}
        In particular, it holds
        \begin{enumerate}[label=(\alph*)]
            \item $\pr_1(\Jeps(x)) = B_\eps(x)$,
            \item $\pr_2(\Jeps(x)) = f(B_\eps(x))$,
            \item $\conv(\pr_3(\Jeps(x))) = \partial_\eps f(x)$.
        \end{enumerate}
    \end{lemma}
    \vspace{-15pt}
    \begin{proof}
        ``$\subseteq$'': Let $(y,\varphi,\xi,\calH) \in \Jeps(x)$ and let $T(\delta)$ be defined as in \eqref{eq:def_T_delta}. Then by definition we can construct sequences $(J^i)_i = (y^i,f(y^i),\nabla f(y^i),\nabla^2 f(y^i))_i$ and $(\delta_i)_i \in \R^{>0}$ with $\delta_i > \eps$, $\delta_i \rightarrow \eps$, $J^i \in T(\delta_i)$ and $J^i \rightarrow (y,\varphi,\xi,\calH)$. Continuity of $f$ implies $\varphi = f(y)$. Furthermore, since $\R^n \setminus \Omega^2 \subseteq \R^n \setminus \Omega$, $\xi \in \partial f(y)$ follows from \eqref{eq:def_clarke_subdifferential}. \\
        ``$\supseteq$'': Let $y \in B_\eps(x)$ and $\xi \in \partial f(y)$. By \cite{C1990}, Theorem 2.5.1, and \ref{enum:A1_2}, there must be a sequence $(y^i)_i \in \R^n \setminus \Omega^2$ with $y^i \rightarrow y$ and $\nabla f(y^i) \rightarrow \xi$. Let $J^i := (y^i,f(y^i),\nabla f(y^i),\nabla^2 f(y^i))$ for $i \in \N$. Then there is some $\eps' > \eps$ such that $J_i \in \J^2_{\eps'} f(x)$ for all $i \in \N$. Since $\J^2_{\eps'} f(x)$ is compact by Lemma \ref{lem:jet_properties}, there is a subsequence $(i_j)_j$ and a matrix $\calH \in \R^{n \times n}$ such that $\lim_{j \rightarrow \infty} J^{i_j} = (y,f(y),\xi,\calH) =: \bar{J}$. By construction it holds $\bar{J} \in \overline{T(\delta)}$ for all $\delta > \eps$, so $\bar{J} \in \Jeps(x)$. \\
        Finally, for (a) and (b) note that the Clarke subdifferential is always nonempty and for (c) recall \eqref{eq:eps_subdiff_via_clarke}. 
    \end{proof}

    Actually evaluating the second-order $\eps$-jet via Definition \ref{def:second_order_eps_jet} is cumbersome due to the closure and the intersection. Fortunately, like for the Clarke subdifferential, there is a simpler representation in case $f$ is defined piecewise. To this end, let $D^2 \subseteq \R^n$ be the set of points at which $f$ is $C^2$ and consider the following class of functions:
    \begin{assum} \label{assum:A2}
        The function $f$ is continuous and there is a finite set $I$ and $C^2$ functions $f_k : \R^n \rightarrow \R$, $k \in I$, called \emph{selection functions}, such that 
        \begin{enumerate}[itemindent=25pt,label=(A2.\arabic*)]
            \item \label{enum:A2_1} $f(x) \in \{ f_k(x) : k \in I\}$ for all $x \in \R^n$ and
            \item \label{enum:A2_2} $D^2 = \R^n \setminus \Omega$, i.e., $f$ is $C^2$ in all points where it is differentiable.
        \end{enumerate}
    \end{assum}

    The condition \ref{enum:A2_1} implies that $f$ is a \emph{piecewise twice differentiable} function in the sense of \cite{S2012}, Section 4.1, and \ref{enum:A2_2} is added for consistency with Assumption \ref{assum:A1}. Before deriving a simpler representation for the $\eps$-jet for this class of functions, we first show that Assumption \ref{assum:A2} implies Assumption \ref{assum:A1}. To this end, let 
    \begin{align*}
        A^e(x) := \left\{ k \in I : x \in \overline{ \{ y \in \R^n : f(y) = f_k(y) \}^\circ } \right\}
    \end{align*}
    be the \emph{essentially active set of} $f$ \emph{at} $x$ (with $S^\circ$ denoting the interior of a set $S \subseteq \R^n$). By the proof of Proposition 4.1.5 in \cite{S2012}, $A^e(x)$ is always nonempty.

    \begin{lemma} \label{lem:PC3_properties}
        Assumption \ref{assum:A2} implies Assumption \ref{assum:A1}.
    \end{lemma}
    \vspace{-15pt}
    \begin{proof}
        \ref{enum:A1_1} follows from \cite{S2012}, Corollary 4.1.1. \ref{enum:A2_2} implies $\Omega = \Omega^2 = \R^n \setminus D^2$, so \ref{enum:A1_2} holds by Rademacher's theorem. To show that \ref{enum:A1_3} holds let $U \subseteq \R^n$ be any bounded set and $\bar{y} \in U \setminus \Omega^2 = U \cap D^2$. Since $A^e(\bar{y}) \neq \emptyset$, there are $\bar{k} \in A^e(\bar{y})$ and a sequence $(y^i)_i \in \{ y \in \R^n : f(y) = f_{\bar{k}}(y) \}^\circ$ with $y^i \rightarrow \bar{y}$. Since $\{ y \in \R^n : f(y) = f_{\bar{k}}(y) \}^\circ$ is open and a subset of $D^2$, we have $\nabla^2 f(y^i) = \nabla^2 f_k(y^i)$ for all $i \in \N$. By continuity of $\nabla^2 f$ and $\nabla^2 f_k$ in $D^2$, this implies that $\nabla^2 f(\bar{y}) = \nabla^2 f_k(\bar{y})$. The proof follows from boundedness of $U$ and continuity of $\nabla^2 f_k$, $k \in I$.
    \end{proof}

    For functions satisfying \ref{enum:A2_1}, the Clarke subdifferential has the well-known representation $\partial f(x) = \conv(\{ \nabla f_k(x) : k \in A^e(x) \})$ (cf.\ \cite{S2012}, Proposition 4.3.1). For the $\eps$-jet we obtain an analogous representation:

    \begin{lemma} \label{lem:jet_for_max_fun}
        Assume that $f$ satisfies Assumption \ref{assum:A2}. Then
        \begin{align} \label{eq:jet_for_max_fun}
            \Jeps(x) = \{ (y,f_k(y),\nabla f_k(y), \nabla^2 f_k(y)) : y \in B_\eps(x), k \in A^e(y) \}
        \end{align}
        for all $x \in \R^n$, $\eps \geq 0$.
    \end{lemma}
    \vspace{-15pt}
    \begin{proof}
        \textbf{Part 1:} Define $T(\delta)$ as in the proof of Lemma \ref{lem:jet_properties}. We first show that
        \begin{align*}
            T(\delta) 
            = \{ (y,f_k(y),\nabla f_k(y), \nabla^2 f_k(y)) : y \in B_\delta(x) \cap D^2, k \in A^e(y) \}
            \quad \forall \delta > 0.
        \end{align*}
        To this end, for $\delta > 0$ let $y \in B_\delta(x) \setminus \Omega^2 = B_\delta(x) \cap D^2$. Since $A^e(y) \neq \emptyset$ there are $k \in A^e(y)$ and a sequence $(y^i)_i \in \{ y \in \R^n : f(y) = f_{k}(y) \}^\circ$ with $y^i \rightarrow y$. Since $\{ y \in \R^n : f(y) = f_{k}(y) \}^\circ$ is open and a subset of $D^2$, we have
        \begin{align*}
            (y^i,f(y^i),\nabla f(y^i), \nabla^2 f(y^i))
            = (y^i,f_k(y^i),\nabla f_k(y^i), \nabla^2 f_k(y^i))
            \quad \forall i \in \N.
        \end{align*}
        By continuity of $f$ and $f_k$ and their derivatives in $D^2$, taking the limit yields
        \begin{align*}
            (y,f(y),\nabla f(y), \nabla^2 f(y))
            = (y,f_k(y),\nabla f_k(y), \nabla^2 f_k(y))
            \quad \forall i \in \N.
        \end{align*}
        \textbf{Part 2:} Proof of ``$\subseteq$'' in \eqref{eq:jet_for_max_fun}: Let $(\bar{y},f(\bar{y}),\bar{\xi},\bar{\calH}) \in \Jeps(x)$. Then there are sequences $(J^i)_i = (y^i,f(y^i),\xi^i,\calH^i)_i$ and $(\delta_i)_i \in \R^{>0}$ with $\delta_i > \eps$, $\delta_i \rightarrow \eps$, $J^i \in T(\delta_i)$ and $J^i \rightarrow (\bar{y},f(\bar{y}),\bar{\xi},\bar{\calH})$. By Part 1 and since $I$ is finite, there must be some $k \in I$ such that $k \in A^e(y^i)$ and $J^i = (y^i,f_k(y^i),\nabla f_k(y^i), \nabla^2 f_k(y^i))$ for infinitely many $i$. Assume w.l.o.g.\ that this holds for all $i \in \N$. By continuity of $f_k$ and its derivatives, it holds
        \begin{align*}
            (\bar{y},f(\bar{y}),\bar{\xi},\bar{\calH})
            = \lim_{i \rightarrow \infty} J^i
            = (\bar{y},f_k(\bar{y}),\nabla f_k(\bar{y}), \nabla^2 f_k(\bar{y})).
        \end{align*}
        Furthermore, from the definition of $A^e$, it is easy to see that $k \in A^e(\bar{y})$. \\
        Proof of ``$\supseteq$'' in \eqref{eq:jet_for_max_fun}:
        Let $\bar{y} \in B_\eps(x)$ and $k \in A^e(\bar{y})$. By definition of $A^e(\bar{y})$ there is a sequence $(y^i)_i$ in the open set $\{ y \in \R^n : f(y) = f_{k}(y) \}^\circ \subseteq D^2$ with $y^i \rightarrow \bar{y}$. By Part 1 and continuity of $f_{k}$ and its derivatives, this shows that $(\bar{y},f_{k}(\bar{y}),\nabla f_{k}(\bar{y}), \nabla^2 f_{k}(\bar{y}))$ is an element of
        \begin{align*}
            \overline{\{ (y,f_{k}(y),\nabla f_{k}(y), \nabla^2 f_{k}(y)) : y \in B_\delta(x) \cap D^2, k \in A^e(y) \}}
            = \overline{T(\delta)}
        \end{align*}
        for all $\delta > \eps$ (since $y^i \in B_\delta(x)$ for $i$ large enough). Since by definition of $T(\delta)$ it holds $\bigcap_{\delta > \eps} \overline{T(\delta)} = \Jeps(x)$, this shows that ``$\supseteq$'' holds in \eqref{eq:jet_for_max_fun}.
    \end{proof}

    We conclude this part with a brief discussion on the computation of the $\eps$-jet for other classes of functions:
    \begin{remark}
        \begin{enumerate}[label=(\alph*)]
            \item It may be possible to weaken \ref{enum:A2_2} in Assumption \ref{assum:A2}. Its purpose is to assure that $\Omega^2$ is a null set for \ref{enum:A1_2}. By \cite{S2012}, Proposition 4.1.5, \ref{enum:A2_1} alone already implies that there is an open and dense subset of $\R^n$ on which $f$ is $C^2$. While this does not imply \ref{enum:A1_2}, a weaker condition than \ref{enum:A2_2} may be sufficient.
            \item For the Clarke subdifferential, the representation in terms of selection functions can be generalized to so-called \emph{lower-}$C^1$ functions (\cite{RW1998}, Definition 10.29): By \cite{RW1998}, Theorem 10.31, for $f(x) = \max_{t \in T} f_t(x)$, it holds
            \begin{align*}
                \partial f(x) = \conv \left( \left\{ \nabla f_t(x) : t \in \argmax_{t \in T} f_t(x) \right\} \right).
            \end{align*}
            Unfortunately, for \emph{lower-}$C^2$ (and thus \emph{lower-}$C^\infty$) functions, an analogous representation of the $\eps$-jet does not to exist. Consider, e.g., the function $f : \R^n \rightarrow\R$, $x \mapsto \| x \|$. Then $f(x) = \max_{t \in T} f_t(x)$ for $f_t(x) := t^\top x$ and $T := \{ y \in \R^n : \| y \| = 1 \}$. Clearly, $\nabla^2 f_t(x) = 0$ for all $t \in T$. However, $\nabla^2 f(x) \neq 0$ for all $x \neq 0$. 
        \end{enumerate}
    \end{remark}

\subsection{Second-order model} \label{subsec:second_order_model}

    In the following, we define and analyze a model for $f$ which is based on the second-order $\eps$-jet. For ease of notation, for $v \in \R^n$ and  $\calH \in \R^{n \times n}$, denote $q_\calH(v) := v^\top \calH v$. The idea is to assign to each element $J \in \Jeps(x)$ a quadratic polynomial via
    \begin{align} \label{eq:Jeps_polynomial}
        J = (y,f(y),\xi,\calH) \mapsto p_J(z) := f(y) + \xi^\top (z-y) + \frac{1}{2} q_\calH(z-y).
    \end{align}
    (If $f$ is $C^2$ around $y$, then this polynomial is simply the second-order Taylor polynomial of $f$ at $y$.) Then, similar to cutting-plane models, we define our model by taking the maximum of these polynomials. More formally, for $x \in \R^n$ and $\eps \geq 0$, we consider the function
    \begin{align} \label{eq:def_Txeps}
        \Txeps(z) 
        := \max_{J \in \Jeps(x)} p_J(z)
        = \max_{(y,f(y),\xi,\calH) \in \Jeps(x)} f(y) + \xi^\top (z-y) + \frac{1}{2} q_\calH(z - y).
    \end{align}
    Due to continuity of \eqref{eq:Jeps_polynomial} (for fixed $z \in \R^n$) and compactness of $\Jeps(x)$ (cf.\ Lemma \ref{lem:jet_properties}), $\Txeps$ is well-defined. Furthermore, it is locally Lipschitz (in fact, lower-$C^\infty$) by \cite{RW1998}, Theorem 10.31. For the remainder of this section, we analyze the modeling error between $\Txeps$ and $f$. To this end, denote
    \begin{align} \label{eq:def_R}
        \Rxeps(z) := \Txeps(z) - f(z).
    \end{align}
    Before we start, it is important to realize that for a function as in Assumption \ref{assum:A1}, we can only hope to obtain a small error in the closed $\eps$-ball in which we evaluate the function $f$ and its derivatives: No matter how we build our model, there could be nonsmooth points arbitrarily close to $B_\eps(x)$ in which $f$ abruptly changes its behavior, causing $|\Rxeps(z)|$ to grow at least linearly with respect to $\| z - x \|$. Thus, we will only derive error estimates for $z \in B_\eps(x)$. First of all, it is easy to see that $\Rxeps(z) \geq 0$, i.e., $\Txeps$ is an overestimate for $f$ on $B_\eps(x)$:
    \begin{lemma} \label{lem:Txeps_overestimate}
        Assume that $f$ satisfies Assumption \ref{assum:A1}. Let $x \in \R^n$ and $\eps \geq 0$. Then $\Rxeps(z) \geq 0$ for all $z \in B_\eps(x)$.
    \end{lemma}
    \vspace{-15pt}
    \begin{proof}
        Let $z \in B_\eps(x)$. By Lemma \ref{lem:jet_projections} there are $\xi_z \in \R^n$ and $\calH_z \in \R^{n \times n}$ such that $(z,f(z),\xi_z,\calH_z) \in \Jeps(x)$. By definition of $\Txeps(z)$, this implies
        \begin{align*}
            \Txeps(z) \geq f(z) + \xi_z^\top (z-z) + \frac{1}{2} q_{\calH_z}(z - z) = f(z).
        \end{align*}
    \end{proof}

    Clearly, when only imposing Assumption \ref{assum:A1}, the overestimation of $\Txeps$ may be quite severe: Consider, e.g., the function $f(x) = -|x|$. Then for $x = 0$ and any $\eps \geq 0$, it holds $\Txeps(z) = |z|$ and $\Txeps(z) - f(z) = 2 |z|$. So even a model that is constantly zero would be a better model for $f$ than $\Txeps$ in this case. Fortunately, however, for function classes that typically occur in optimization, the model satisfies better estimates for the error. First of all, consider the following class:

    \begin{assum} \label{assum:A3}
        Assume that $f$ is convex and satisfies Assumption \ref{assum:A1}.
    \end{assum}
    
    For such convex functions, we obtain the following result:

    \begin{lemma} \label{lem:error_convex}
        Assume that $f$ satisfies Assumption \ref{assum:A3}. Then for every bounded set $U \subseteq \R^n$ and every $\eps' > 0$, there is some $K > 0$ such that
        \begin{align*}
            \max_{z \in B_\eps(x)} \Rxeps(z) \leq K \eps^2 \quad \forall x \in U, \eps \in (0,\eps'].
        \end{align*}
    \end{lemma}
    \vspace{-15pt}
    \begin{proof}
        Let $U \subseteq \R^n$ be bounded and $\eps' > 0$. Then
        \begin{align*}
            \Txeps(z) - f(z)
            &= \max_{(y,f(y),\xi,\calH) \in \Jeps(x)} f(y) - f(z) + \xi^\top (z-y) + \frac{1}{2} q_\calH(z - y) \\
            &\leq \max_{(y,f(y),\xi,\calH) \in \Jeps(x)} \frac{1}{2} (z - y)^\top \calH (z - y) \\
            &\leq \max_{(y,f(y),\xi,\calH) \in \Jeps(x)} 2 \| \calH \| \eps^2
            \quad \forall x \in U, \eps \in (0,\eps'], z \in B_\eps(x),
        \end{align*}
        where $\| \calH \|$ denotes the Frobenius norm of $\calH$ and $f(y) - f(z) + \xi^\top (z-y) \leq 0$ follows from convexity of $f$. Since $U + B_{\eps'}(0)$ is bounded (with $+$ denoting the Minkowski sum of sets), there are $\bar{x} \in U$ and $\bar{\eps} > 0$ such that $U + B_{\eps'}(0) \subseteq B_{\bar{\eps}}(\bar{x})$. Compactness of $\mathcal{J}_{\bar{\eps}}^2 f(\bar{x})$ shows that $\| \calH \|$ in the above inequality is bounded, completing the proof.
    \end{proof}

    The proof of the previous lemma shows that if we would set $\calH = 0$ in the definition of $\Txeps$, then the model would exactly equal $f$ on $B_\eps(x)$ if $f$ is convex. This is no surprise, as this property is well-known for cutting-plane models. So in a sense, adding second-order information is actually a hindrance in the convex case. However, this property requires the entire $\eps$-jet to be available, which is only the case theory. In practice, we will merely be able to consider a finite subset of $\Jeps(x)$ corresponding to a finite number of sample points from $B_\eps(x)$. In this case, when compared to a cutting-plane model, the second-order information gives us the chance to have a smaller error at points where we did not sample an element of the $\eps$-jet.

    The second class of functions we consider, which is often encountered in nonsmooth optimization, is the class of max-type functions:

    \begin{assum} \label{assum:A4}
        Assume that $f$ satisfies Assumption \ref{assum:A2} with all $f_k$ being $C^3$ and
        \begin{align*}
            f(x) = \max_{k \in I} f_k(x) \quad \forall x \in \R^n.
        \end{align*}
    \end{assum}

    Since functions satisfying Assumption \ref{assum:A4} behave like $C^3$ functions on certain open subsets of $\R^n$, we can exploit the error estimate we get from the second-order Taylor expansion there, giving us the following result:

    \begin{lemma} \label{lem:error_max_fun}
        Assume that $f$ satisfies Assumption \ref{assum:A4}. Then for every bounded set $U \subseteq \R^n$ and every $\eps' > 0$, there is some $K > 0$ such that 
        \begin{align*}
            \max_{z \in B_\eps(x)} \Rxeps(z)
            \leq K \eps^{3}
            \quad \forall x \in U, \eps \in (0,\eps'].
        \end{align*}
    \end{lemma}
    \vspace{-15pt}
    \begin{proof}
        Let $U \subseteq \R^n$ be bounded and $\eps' > 0$. Assume w.l.o.g.\ that $U$ is convex. \\
        \noindent\textbf{Part 1:} Let $k \in I$ and $y \in U + B_{\eps'}(0)$. Denote 
        \begin{align*}
            J_{k,y} := (y, f_k(y), \nabla f_k(y), \nabla^2 f_k(y)).
        \end{align*}
        Taylor's theorem (with higher-order derivatives in the notation of \cite{K2004}, p.\ 65) applied to $f_k$ shows that for any $z \in U + B_{\eps'}(0)$ there is some $a \in \conv(\{y,z\})$ such that
        \begin{align*}
            f_k(z) - p_{J_{k,y}}(z)
            = \frac{1}{6} \deriv^{(3)} f_k(a)(z - y)^3
        \end{align*}
        with $p_{J_{k,y}}$ as in \eqref{eq:Jeps_polynomial}. By continuity of $\deriv^{(3)} f_k$, finiteness of $|I|$ and boundedness of $U + B_{\eps'}(0)$, there is an upper bound for the right-hand side of this equality that does not depend on $k$ and $a$. More formally, there is some $K' > 0$ such that
        \begin{align} \label{eq:proof_lem_Taylor_max_fun_1}
            |f_k(z) - p_{J_{k,y}}(z)| 
            \leq \frac{K'}{6} \| z - y \|^3 \quad \forall y, z \in U + B_{\eps'}(0), k \in I.
        \end{align}
        Let $x \in U$ and $\eps \in (0,\eps']$. Then \eqref{eq:proof_lem_Taylor_max_fun_1} implies
        \begin{align} \label{eq:proof_lem_Taylor_max_fun_2}
            |f_k(z) - p_{J_{k,y}}(z)| \leq K \eps^3 \quad \forall y, z \in B_\eps(x), k \in I
        \end{align}
        for $K := (4 K')/3$, since $\| z - y \| \leq 2 \eps$. \\
        \textbf{Part 2:} Let $x \in U$,  $\eps \in (0,\eps']$ and $z \in B_\eps(x)$. By Lemma \ref{lem:jet_for_max_fun} it holds $\Txeps(z) = p_{J_{\bar{k},\bar{y}}}(z)$ for some $\bar{y} \in B_\eps(x)$, $\bar{k} \in A^e(\bar{y})$. Now $f(z) = \max_{k \in I} f_k(z) \geq f_{\bar{k}}(z)$ and \eqref{eq:proof_lem_Taylor_max_fun_2} imply that
        \begin{align*}
            \Rxeps(z)
            = \Txeps(z) - f(z) 
            = p_{J_{\bar{k},\bar{y}}}(z) - f(z) 
            \leq p_{J_{\bar{k},\bar{y}}}(z) - f_{\bar{k}}(z) 
            \leq K \eps^3,
        \end{align*}
        completing the proof.
    \end{proof}

    It is important to note that for a cutting-plane model (where $\calH = 0$ in \eqref{eq:def_Txeps}), the cubic error estimate in the previous lemma does not hold. (Consider, e.g., $f(x) = -x^2$ with $x = 0$.) As such, this result quantifies the benefit of using second-order information. We conclude this section with some remarks on the model $\Txeps$:
    
    \begin{remark}
        \begin{enumerate}[label=(\alph*)]
            \item It is worth pointing out that $\Txeps$ is not a first- or second-order model in the sense of \cite{N2010}, Definition 1 and 3 (with $\phi(\cdot,x) = \Txeps(\cdot)$), as in general, $\Txeps$ may be nonconvex and we may have $\Txeps(x) \neq f(x)$ and $\partial \Txeps \nsubseteq \partial f(x)$.
            \item Considering the proof of Lemma \ref{lem:error_max_fun}, it appears that it could be generalized to $\Txeps$ being the maximum of Taylor expansions of any order $q \in \N$, resulting in the upper bound $K \eps^{q+1}$ for the error.
        \end{enumerate}
    \end{remark}

\section{Descent method} \label{sec:descent_method}

In this section, we construct a descent method for minimizing $f$ based on minimizing the second-order model $\Txeps$ from the previous section. We begin by deriving an abstract version of this method (Algo.\ \ref{algo:abstract_descent_method}) in Section \ref{subsec:abstract_algorithm}, for which we assume that the entire second-order $\eps$-jet is available at each $x \in \R^n$. We prove convergence to critical points for the case where $f$ is convex or of max-type. Since the entire $\eps$-jet is not available in practice, we then construct an approximation scheme for $\Jeps(x)$ (Algo.\ \ref{algo:approx_jet}) in Section \ref{subsec:approximating_eps_jet} that only requires a single element of $\J_0^2 f(y)$ at every $y \in B_\eps(x)$ and prove its termination. In Section \ref{subsec:practical_algorithm}, we insert this scheme into the abstract algorithm, yielding the practical Algo.\ \ref{algo:practical_descent_method}, for which we then prove the same convergence results as for the abstract method.

\subsection{Abstract algorithm} \label{subsec:abstract_algorithm}

    The idea of our method is to minimize $f$ by iteratively minimizing the model $\Txeps$ from \eqref{eq:def_Txeps} for varying $x \in \R^n$ and $\eps > 0$. However, since $\Txeps$ may be unbounded below in case $f$ is nonconvex, we cannot minimize it over $\R^n$. Instead, as we can only guarantee a certain quality of $\Txeps$ on $B_\eps(x)$ (cf.\ Section \ref{subsec:second_order_model}), we consider the subproblem
    \begin{align} \label{eq:min_Txeps}
        \min_{z \in B_\eps(x)} \Txeps(z) = \min_{z \in B_\eps(x)} \max_{(y,f(y),\xi,\calH) \in \Jeps(x)} f(y) + \xi^\top (z-y) + \frac{1}{2} q_\calH(z - y).
    \end{align}
    This problem is well-defined since $\Txeps$ is continuous. Let
    \begin{align*}
        \bar{z}(x,\eps) \in \argmin_{z \in B_\eps(x)} \Txeps(z)
    \end{align*}
    be an arbitrary solution of \eqref{eq:min_Txeps} and let $\theta(x,\eps) = \Txeps(\bar{z}(x,\eps))$ be the optimal value. For the sake of readability, we will omit the dependency of $\theta(x,\eps)$ and $\bar{z}(x,\eps)$ on $x$ and $\eps$ whenever the context allows it.
    
    The following lemma yields an estimate for how much decrease (if any) we can expect when moving from $x$ to $\bar{z}(x,\eps)$. For a compact set $S \subseteq \R^n$ denote $\min(\| S \|) := \min_{s \in S} \| s \|$.

    \begin{lemma} \label{lem:model_descent_estimate}
        Assume that $f$ satisfies Assumption \ref{assum:A1}. Let $x \in \R^n$ and $\eps \geq 0$.
        \begin{enumerate}[label=(\alph*)]
            \item It holds $f(\bar{z}) \leq \theta$.
            \item Let $z^* \in \argmin_{z \in B_\eps(x)} f(z)$. Then 
                \begin{align*}
                    \theta - f(x) 
                    \leq - \eps \min(\| \partial_\eps f(x) \|) + R_{x,\eps}(z^*).
                \end{align*}
        \end{enumerate}
    \end{lemma}
    \vspace{-5pt}
    \begin{proof}
        (a) By Lemma \ref{lem:Txeps_overestimate} it holds $f(\bar{z}) \leq \Txeps(\bar{z}) = \theta$. \\
        (b) Let $\bar{v} := -\argmin_{\xi \in \partial_\eps f(x)} \| \xi \|$. Then
        \begin{align*}
            \theta 
            = \Txeps(\bar{z}) 
            \leq \Txeps(z^*) 
            = f(z^*) + R_{x,\eps}(z^*)
            \leq f \left( x + \eps \frac{\bar{v}}{\| \bar{v} \|} \right) + R_{x,\eps}(z^*).
        \end{align*}
        Using \eqref{eq:eps_subdiff_decrease} (for $t = \eps / \| \bar{v} \|$) we obtain
        \begin{align*}
            \theta - f(x)  
            \leq - \eps \| \bar{v} \| + R_{x,\eps}(z^*), 
        \end{align*}
        which completes the proof.
    \end{proof}

    Considering the estimates for $\Rxeps$ in Lemma \ref{lem:error_convex} and Lemma \ref{lem:error_max_fun}, this leads to the following result:

    \begin{lemma} \label{lem:model_critical}
        Assume that $f$ satisfies Assumption \ref{assum:A1} and that there are $K > 0$, $U \subseteq \R^n$ and $\eps' > 0$ such that $\Rxeps(z) \leq K \eps^2$ for all $x \in U$, $z \in B_\eps(x)$, $\eps \in (0,\eps']$.
        \begin{enumerate}[label=(\alph*)]
            \item Let $x \in U$. Then
                \begin{align*}
                    \limsup_{\eps \searrow 0} \frac{\theta(x,\eps) - f(x)}{\eps} \leq -\min(\| \partial f(x) \|).
                \end{align*}
            \item If there are $(x^j)_j \in U$ and $(\eps_j)_j \in (0,\eps']$ such that $x^j \rightarrow x^* \in U$, $\eps_j \rightarrow 0$ and
                \begin{align*}
                    \liminf_{j \rightarrow \infty} \frac{\theta(x^j,\eps_j) - f(x^j)}{\eps_j} \geq 0,
                \end{align*}
                then $x^*$ is a critical point of $f$.
        \end{enumerate}
    \end{lemma}
    \vspace{-10pt}
    \begin{proof}
        (a) Let $x \in U$. By Lemma \ref{lem:model_descent_estimate} we have
        \begin{align*}
            \frac{\theta(x,\eps) - f(x)}{\eps} 
            &\leq -\min(\| \partial_\eps f(x) \|) + \frac{R_{x,\eps}(z^*)}{\eps} \\
            &\leq -\min(\| \partial_\eps f(x) \|) + K \eps
            \quad \forall \eps \in (0,\eps'].
        \end{align*}
        For $\eps \searrow 0$, by \cite{C1990}, Proposition 2.1.5(b), the cluster points of any sequence of convex combinations of elements of $\cup_{y \in B_\eps(x)} \partial f(y)$ lie in $\partial f(x)$. Combined with the fact that $\min(\| \partial_\eps f(x) \|) \leq \min(\| \partial f(x) \|)$ for all $\eps > 0$, we obtain
        \begin{align*}
            \lim_{\eps \searrow 0} \min(\| \partial_\eps f(x) \|) = \min(\| \partial f(x) \|),
        \end{align*}
        which completes the proof.\\
        (b) Analogous to (a) we have
        \begin{align*}
            \frac{\theta(x^j,\eps_j) - f(x^j)}{\eps_j}
            \leq -\min(\| \partial_{\eps_j} f(x^j) \|) + K \eps_j
            \quad \forall j \in \N.
        \end{align*}
        Again by \cite{C1990}, Proposition 2.1.5(b), for $j \rightarrow \infty$, the cluster points of any sequence of convex combinations of elements of $\cup_{y \in B_{\eps_j}(x^j)} \partial f(y)$ lie in $\partial f(x^*)$. (Note that in contrast to (a), the $x$ also varies now.) This yields 
        \begin{align*}
            \liminf_{j \rightarrow \infty} \left( -\min(\| \partial_{\eps_j} f(x^j) \|) \right)
            = - \limsup_{j \rightarrow \infty} \min(\| \partial_{\eps_j} f(x^j) \|)
            \leq -\min(\| \partial f(x^*) \|).
        \end{align*}
        Thus, we obtain
        \begin{align*}
            0 
            \leq \liminf_{j \rightarrow \infty} \frac{\theta(x^j,\eps_j) - f(x^j)}{\eps_j}
            \leq -\min(\| \partial f(x^*) \|),
        \end{align*}
        so $\min(\| \partial f(x^*) \|) = 0$, which completes the proof.
    \end{proof}

    Combination of Lemma \ref{lem:error_convex} and Lemma \ref{lem:error_max_fun} with Lemma \ref{lem:model_descent_estimate}(a) and Lemma \ref{lem:model_critical}(a) shows that if $f$ is convex or of max-type and $x$ is not a critical point, then $f(\bar{z})$ is smaller than $f(x)$ for $\eps > 0$ small enough. In other words, as long as $x$ is not already critical, we can decrease the value of $f$ by solving the subproblem \eqref{eq:min_Txeps} for $\eps$ small enough. This motivates the following strategy for our descent method: Given an initial point $x^0$ and sequences $(\eps_j)_j, (\tau_j)_j \in \R^{>0}$ with $\eps_j \rightarrow 0$ and $\tau_j \rightarrow 0$, we first generate a sequence $(x^{1,i})_i$ via $x^{1,0} = x^0$ and
    \begin{align*}
        x^{1,i+1} = \bar{z}(x^{1,i},\eps_1) \quad \forall i \in \N \cup \{ 0 \}.
    \end{align*}
    Then by Lemma \ref{lem:model_descent_estimate}(a),
    \begin{align*}
        f(x^{1,i+1}) - f(x^{1,i}) \leq \theta(x^{1,i},\eps_1) - f(x^{1,i}) \quad \forall i \in \N \cup \{ 0 \}.
    \end{align*}
    If $f$ is bounded below, then the right-hand side of this inequality must eventually be arbitrarily close to $0$ or positive. In particular, there must be some $i' \in \N \cup \{ 0 \}$ with
    \begin{align*}
        \frac{\theta(x^{1,i'},\eps_1) - f(x^{1,i'})}{\eps_1} > -\tau_1,
    \end{align*}
    which means that the predicted decrease is less than $\eps_1 \tau_1$. When this occurs we choose the next (potentially smaller) tolerances $\eps_2$ and $\tau_2$ and start a new sequence $(x^{2,i})_i$ via $x^{2,0} = x^{1,i'}$ and $x^{2,i+1} = \bar{z}(x^{2,i},\eps_2)$ and so on. The resulting method is Algo.\ \ref{algo:abstract_descent_method}.

    \begin{algorithm} 
        \caption{Abstract descent method}
        \label{algo:abstract_descent_method}
       	\begin{algorithmic}[1] 
            \Require Initial point $x^0 \in \R^n$, vanishing sequences $(\eps_j)_j$, $(\tau_j)_j \in \R^{>0}$.
       		\State Initialize $j = 1$, $i = 0$ and $x^{1,0} = x^0$.
       		\State Compute $\bar{z}(x^{j,i},\eps_j)$ and $\theta(x^{j,i},\eps_j)$. \Comment{Solve subproblem}
            \If{$(\theta(x^{j,i},\eps_j) - f(x^{j,i})) / \eps_j > -\tau_j$} \Comment{Check decrease}
                \State Set $x^{j+1,0} = x^{j,i}$, $j = j+1$ and $i = 0$. \Comment{Change tolerances}
            \Else
                \State Set $x^{j,i+1} = \bar{z}(x^{j,i},\eps_j)$ and $i = i+1$. \Comment{Descent step}
            \EndIf
            \State Go to Step 2.
        \end{algorithmic}
    \end{algorithm}	

    For $j \in \N$ let $N_j \in \N \cup \{ 0, \infty \}$ be the final $i$ in Step 4 of Algo.\ \ref{algo:abstract_descent_method}, i.e., $N_j$ is the number of descent steps that are performed for each fixed $j$. For ease of notation let
    \begin{align} \label{eq:def_xj}
        x^j := x^{j,N_j} = x^{j+1,0} \quad \forall j \in \N.    
    \end{align}
    (If $N_j = \infty$ for some $j \in \N$, then $(x^j)_j$ is a finite sequence.) Furthermore, denote by $(\hat{x}^l)_l$ the entire sequence that is generated, i.e.,
    \begin{align} \label{eq:def_xhatl}
        (\hat{x}^l)_l = (x^{1,0},x^{1,1},\dots,x^{1,N_1},x^{2,0},x^{2,1},\dots,x^{2,N_2},\dots).
    \end{align}
    By construction we have $f(\hat{x}^{l+1}) \leq f(\hat{x}^l)$ for all $l \in \N$. From Lemma \ref{lem:model_descent_estimate} and Lemma \ref{lem:model_critical}, we obtain the following convergence result:
    
    \begin{theorem} \label{thm:algo1_convergence}
        Assume that $f$ satisfies Assumption \ref{assum:A3} or Assumption \ref{assum:A4}. Let $(x^j)_j$ and $(\hat{x}^j)_j$ be the sequences generated by Algo.\ \ref{algo:abstract_descent_method}.
        If the level set $\{ y \in \R^n : f(y) \leq f(x^0) \}$ is bounded, then $(x^j)_j$ has an accumulation point $x^* \in \R^n$ and all accumulation points of $(x^j)_j$ are critical. Furthermore, $f(\hat{x}^l) \rightarrow f(x^*)$.
    \end{theorem}
    \vspace{-15pt}
    \begin{proof}
        Let $U := \{ y \in \R^n : f(y) \leq f(x^0) \}$. Since $(f(x^j))_j$ is non-increasing we have $(x^j)_j \in U$. Boundedness of $U$ and continuity of $f$ imply compactness of $U$, so $(x^j)_j$ must have an accumulation point $x^* \in U$. Since $(f(\hat{x}^l))_l$ is non-increasing we must have $f(\hat{x}^l) \rightarrow f(x^*)$. By Lemma \ref{lem:model_descent_estimate}(a) and the condition in Step 3, for all descent steps we have
        \begin{align*}
            f(x^{j,i+1}) - f(x^{j,i})
            \leq \theta^{j,i} - f(x^{j,i})
            \leq - \eps_j \tau_j \quad \forall j \in \N, i \in \{0,\dots,N_j-1\}.
        \end{align*}
        Thus, since $f$ is bounded below in $U$, all $N_j$ have to be finite. This means that the condition in Step 3 has to hold infinitely many times, i.e.,
        \begin{align*}
            -\tau_j
            < \frac{\theta(x^{j,N_j},\eps_j) - f(x^{j,N_j})}{\eps_j}
            &= \frac{\theta(x^j,\eps_j) - f(x^j)}{\eps_j}
            \quad \forall j \in \N.
        \end{align*}
        Since $\tau_j \rightarrow 0$, this implies
        \begin{align*}
            \liminf_{j \rightarrow \infty} \frac{\theta(x^j,\eps_j) - f(x^j)}{\eps_j}
            \geq 0.
        \end{align*}
        By assumption, we can apply Lemma \ref{lem:error_convex} or Lemma \ref{lem:error_max_fun} to see that there are $K > 0$ and $\eps' > 0$ with $\Rxeps(z) \leq K \eps^2$ for all $x \in U$, $z \in B_\eps(x)$, $\eps \in (0,\eps']$. Application of Lemma \ref{lem:model_critical}(b) completes the proof.
    \end{proof}

    We conclude the discussion of Algo.\ \ref{algo:abstract_descent_method} with the following remark:
    \begin{remark}
        \begin{enumerate}[label=(\alph*)]
            \item For convex quadratic functions, Newton's method finds a minimum in a single iteration. Analogously, Algo.\ \ref{algo:abstract_descent_method} finds a minimum (if one exists) in a single iteration for piecewise-quadratic max-type functions, i.e., functions that satisfy Assumption \ref{assum:A4} with quadratic selection functions $f_k$, $k \in I$ (if $\eps_1$ is large enough for $B_{\eps_1}(x^0)$ to contain a minimum).
            \item In the subproblem \eqref{eq:min_Txeps}, $B_\eps(x)$ could be interpreted as a ``trust-region'' for the model $\Txeps$, and Algo.\ \ref{algo:abstract_descent_method} could be regarded as a type of \emph{trust-region method} \cite{CGT2000}. The main differences to a standard trust-region method (see, e.g., Algorithm 6.1.1 in \cite{CGT2000}) are the acceptance criterion in Step 3 and the fact that in Algo.\ \ref{algo:abstract_descent_method}, there is no mechanism for increasing the ``trust-region radius'' $\eps$ in case $\Txeps$ is a sufficiently good model in $B_\eps(x)$.
        \end{enumerate}
    \end{remark}

\subsection{Approximating the \texorpdfstring{\boldmath{$\eps$}}{epsilon}-jet} \label{subsec:approximating_eps_jet}

    The assumption in the previous section that we are able to evaluate the entire second-order $\eps$-jet means that Algo.\ \ref{algo:abstract_descent_method} is not implementable in practice. For first-order information, the standard oracle assumption for the Clarke subdifferential $\partial f$ is that at every $x \in \R^n$, we are able to compute a single, arbitrary element of $\partial f(x) = \partial_0 f(x)$ (cf.\ \cite{L1989}, (1.10)). As such, for the $\eps$-jet, we make the analogous assumption that we can compute a single, arbitrary element of $\J_0^2 f(x)$ at every $x \in \R^n$. (Due to \ref{enum:A1_2}, for almost all $x \in \R^n$, this means that we have to be able to compute the objective value, the gradient and the Hessian at $x$.) In the following, we show how this assumption allows us to obtain a finite approximation of $\Jeps(x)$ (for $\eps > 0$) that is good enough to be used instead of $\Jeps(x)$ in Algo.\ \ref{algo:abstract_descent_method}.

    For $x \in \R^n$ and $\eps > 0$ let $W \subseteq \Jeps(x)$ be a finite subset. Consider the model
    \begin{align} \label{eq:def_TxepsW}
        \Txeps^W(z) := \max_{(y,f(y),\xi,\calH) \in W} f(y) + \xi^\top (z-y) + \frac{1}{2} q_\calH(z - y)
    \end{align}
    and the subproblem
    \begin{align} \label{eq:min_TxepsW}
        \min_{z \in B_\eps(x)} \Txeps^W(z).
    \end{align}
    Let $\bar{z}^W(x,\eps) \in \argmin_{z \in B_\eps(x)} \Txeps^W(z)$ be a solution and $\theta^W(x,\eps) = \Txeps^W(\bar{z}^W(x,\eps))$ be the optimal value of this subproblem. (As before, we will omit the dependency on $x$ and $\eps$ whenever the context allows it.)

    \begin{remark}
        The subproblem \eqref{eq:min_TxepsW} has the equivalent epigraph formulation
        \begin{equation} \label{eq:min_TxepsW_quad}
            \begin{aligned}
                \min_{z \in \R^n, \beta \in \R} \ & \beta \\
                \text{s.t.} \ & f(y) + \xi^\top (z-y) + \frac{1}{2} (z - y)^\top \calH (z - y) \leq \beta   \quad \forall (y,f(y),\xi,\calH) \in W,\\
                & \| z - x \|^2 \leq \eps^2,
            \end{aligned}
        \end{equation}
        which has a linear objective and (possibly nonconvex) quadratic constraints. In contrast to the subproblems that occur in the bundle method (Problem (12.3) in \cite{BKM2014}) or the gradient sampling method (Problem (6.5) in \cite{BCL2020}), \eqref{eq:min_TxepsW_quad} is not a quadratic problem. Instead, it belongs to the more general class of nonconvex quadratically constrained quadratic programs (\cite{BV2004}, Section 4.4).
    \end{remark}

    By definition of $\Jeps(x)$ we have $\J_0^2 f(y) \subseteq \Jeps(x)$ for all $y \in B_\eps(x)$. Thus, we could take the route of classic gradient sampling (cf.\ \cite{BLO2005,BCL2020}) and simply approximate $\Jeps(x)$ by generating random points $y \in B_\eps(x)$ and evaluating $\J_0^2 f(y)$. However, even for first-order information, it is known that random sampling requires a large number of samples to robustly approximate the $\eps$-subdifferential. Since in our case, we also have to evaluate second-order derivatives in every sample point, we want to keep the number of sample points as low as possible. Therefore, we will instead take a deterministic approach inspired by \cite{MY2012,GP2021} that tries to only generate sample points that are actually relevant. (Additionally, this means that we do not have to rely on stochastic arguments when analyzing the convergence of the resulting method.)

    To this end, assume that $x \in \R^n$, $\eps > 0$ and $\tau > 0$ are given. The idea is to start with an initial approximation $W \subseteq \Jeps(x)$ consisting of finitely many elements 
    (e.g., $W = \{ (x,f(x),\xi,\calH) \} \subseteq \J_0^2 f(x)$)
    and to then iteratively add more elements to $W$ until it is a sufficient approximation of $\Jeps(x)$. In the following, we define what it means for $W$ to be a ``sufficient'' approximation and construct a way to add elements to $W$ that improve the approximation quality of $\Txeps^W$ if it is insufficient.
    First of all, by definition of $\theta$ and $\theta^W$, we have $\theta \geq \theta^W$ for any $W \subseteq \Jeps(x)$. Therefore, if
    \begin{align} \label{eq:thetaW_stop1}
        \frac{\theta^W - f(x)}{\eps} > -\tau,
    \end{align}
    then also the condition in Step 3 of Algo.\ \ref{algo:abstract_descent_method} holds and we can stop the approximation scheme (and increase $j$). Otherwise, we know that
    \begin{align} \label{eq:thetaW_decrease}
        \theta^W - f(x) \leq -\tau \eps.
    \end{align}
    For the exact $\theta$, this inequality would imply a decrease in the objective value by Lemma \ref{lem:model_descent_estimate}(a). However, Lemma \ref{lem:model_descent_estimate}(a) does in general not apply to the approximation $\theta^W$, since its proof is based on having $\pr_1(\Jeps(x)) = B_\eps(x)$. Due to this, we have to manually check whether we have decrease in the objective value when moving from $x$ to $\bar{z}^W$. To this end, we weaken the inequality in Lemma \ref{lem:model_descent_estimate}(a) to
    \begin{align} \label{eq:thetaW_stop2}
        f(\bar{z}^W) \leq c \theta^W + (1 - c) f(x)
    \end{align}
    for $c \in (0,1)$. Combined with \eqref{eq:thetaW_decrease}, this leads to
    \begin{align} \label{eq:fzW_decrease}
        f(\bar{z}^W) 
        \leq c (f(x) - \tau \eps) + (1 - c) f(x)
        = f(x) - c \tau \eps.
    \end{align}
    In words, if \eqref{eq:thetaW_stop1} is violated, then \eqref{eq:thetaW_stop2} implies that we have at least $c$ percent of the decrease we could expect when using the full $\eps$-jet. Therefore, we consider a subset $W \subseteq \Jeps(x)$ to be a sufficient approximation if \eqref{eq:thetaW_stop1} or \eqref{eq:thetaW_stop2} hold. For adding new elements to $W$ in case it is not sufficient, consider the following lemma:
    \begin{lemma} \label{lem:new_jet_element}
        Assume that $f$ satisfies Assumption \ref{assum:A1}. Let $x \in \R^n$, $\eps > 0$, $W \subseteq \Jeps(x)$ and $c \in (0,1)$. If both \eqref{eq:thetaW_stop1} and \eqref{eq:thetaW_stop2} are violated, then $\bar{z}^W \notin \pr_1(W)$.
    \end{lemma}
    \vspace{-15pt}
    \begin{proof}
        Inequality \eqref{eq:thetaW_stop1} being violated implies that $f(x) \geq \theta^W + \tau \eps$. With \eqref{eq:thetaW_stop2} being violated as well, it follows that
        \begin{equation} \label{eq:proof_new_jet_element}
            \begin{aligned}
                f(\bar{z}^W)
                &> c \theta^W + (1 - c) f(x)
                \geq c \theta^W + (1 - c)(\theta^W + \tau \eps) \\
                &= \theta^W + (1 - c) \tau \eps
                = \TWxeps(\bar{z}^W) + (1 - c) \tau \eps.
            \end{aligned}
        \end{equation}
        Since by definition of $\TWxeps$ it holds $\TWxeps(y) \geq f(y)$ for all $y \in \pr_1(W)$, this completes the proof. 
    \end{proof}

    By the previous lemma, in case $W$ is insufficient, an element of $\Jeps(x) \setminus W$ can be computed by evaluating $\J_0^2 f$ in a solution $\bar{z}^W$ of the subproblem \eqref{eq:min_TxepsW}.
    \begin{algorithm} 
        \caption{Sampling scheme for second-order $\eps$-jet}
        \label{algo:approx_jet}
        \begin{algorithmic}[1] 
        	\Require Point $x \in \R^n$, approximation parameter $c \in (0,1)$, radius $\eps > 0$, improvement tolerance $\tau > 0$.
        	\State Sample $J_{x} = (x,f(x),\xi_{x},\calH_{x}) \in \J_0^2 f(x)$ and set $W = \{ J_{x} \}$.
        	\State Compute $\theta^W = \theta^W(x,\eps)$ and $\bar{z}^W = \bar{z}^W(x,\eps)$ via \eqref{eq:min_TxepsW} (or \eqref{eq:min_TxepsW_quad}).
            \If{$(\theta^W - f(x)) / \eps > -\tau$ or $f(\bar{z}^W) \leq c \theta^W + (1 - c) f(x)$}
                \State Stop.
            \EndIf
            \State Sample $J_{\bar{z}^W} = (\bar{z}^W,f(\bar{z}^W),\xi_{\bar{z}^W},\calH_{\bar{z}^W}) \in \J_0^2 f(\bar{z}^W)$, set $W = W \cup \{ J_{\bar{z}^W} \}$ and go to Step 2.
        \end{algorithmic} 
    \end{algorithm}
    The above considerations motivate Algo.\ \ref{algo:approx_jet} for the approximation of $\Jeps(x)$. By construction, the method may only stop when one of the inequalities in Step 3 holds. If the first inequality holds, then $\eps$ and $\tau$ have to be changed, and if the second inequality holds (and the first one does not hold), then $\bar{z}^W$ yields sufficient decrease. The following theorem shows that one of these inequalities must hold after a finite number of iterations, i.e., it shows that the algorithm terminates:

    \begin{theorem} \label{thm:algo2_termination}
        Assume that $f$ satisfies Assumption \ref{assum:A1}. Then Algo.\ \ref{algo:approx_jet} terminates.
    \end{theorem}
    \vspace{-15pt}
    \begin{proof}
        The idea is to show that $\bar{z}^W$ computed in Step 2 has a distance $r > 0$ to all points in $\pr_1(W)$ which does not depend on $\bar{z}^W$ or $W$. By compactness of $B_\eps(x)$, this shows that only finitely many of such $\bar{z}^W$ can be added to $W$, so the algorithm has to terminate.\\
        \textbf{Part 1:} Assume that neither of the conditions in Step 3 hold. Then, as in the proof of Lemma \ref{lem:new_jet_element}, we have $f(\bar{z}^W) > \Txeps^W(\bar{z}^W) + (1 - c) \tau \eps$ which is equivalent to
        \begin{align} \label{eq:proof_thm_algo2_termination}
            \bar{z}^W \in \{ z \in B_\eps(x) : f(z) - \Txeps^W(z) > (1-c) \tau \eps \}.
        \end{align}
        Furthermore $f(z') - \Txeps^W(z') \leq 0$ for all $z' \in \pr_1(W)$. \\
        \textbf{Part 2:} By \cite{RW1998}, Theorem 9.2 and Theorem 10.31, $\Txeps^W$ is Lipschitz continuous on $B_\eps(x)$, and a Lipschitz constant $L'$ is given by
        \begin{align*}
            L'
            := \sup_{z \in B_\eps(x)} \max_{J \in \Jeps(x)} \| \nabla p_J(z) \|
            \geq \sup_{z \in B_\eps(x)} \max_{J \in W} \| \nabla p_J(z) \|,
        \end{align*}
        where $\nabla p_J(z) = \xi + \frac{1}{2}(\calH + \calH^\top) (z - y)$ for $J = (y,f(y),\xi,\calH)$. Note that $L'$ does not depend on $W$ and that it is finite by boundedness of $\Jeps(x)$.\\
        \textbf{Part 3:} Let $L$ be a Lipschitz constant of $f$ on $B_\eps(x)$. Then $L + L'$ is a Lipschitz constant of $z \mapsto f(z) - \Txeps^W(z)$ on $B_\eps(x)$. Let $r := ((1 - c) \tau \eps)/(L + L')$.
        Then by \eqref{eq:proof_thm_algo2_termination}, we have
        \begin{align*}
            (1 - c) \tau \eps - (f(z) - \Txeps^W(z))
            &< f(\bar{z}^W) - \Txeps^W(\bar{z}^W) - (f(z) - \Txeps^W(z)) \\
            &\leq (L + L') \| \bar{z}^W - z \| 
            \leq (1 - c) \tau \eps
        \end{align*}
        for all $z \in B_r(\bar{z}^W) \cap B_\eps(x)$. In particular, 
        \begin{align*}
            f(z) - \Txeps^W(z) > 0 \quad \forall z \in B_r(\bar{z}^W) \cap B_\eps(x),
        \end{align*}
        which implies that $B_r(\bar{z}^W) \cap \pr_1(W) = \emptyset$ by Part 1. In other words, $\bar{z}^W$ has a distance of at least $r$ to all previous sample points and $r$ does not depend on $\bar{z}^W$ and $W$. \\
        \textbf{Part 4:} If the algorithm would not terminate, then there would be an infinite sequence of points $\bar{z}^W$ added to $W$ in Step 6. By compactness of $B_\eps(x)$, this sequence would have an accumulation point. This contradicts Part 3, which completes the proof.
    \end{proof}

    We conclude the discussion of Algo.\ \ref{algo:approx_jet} with a remark on the initialization in Step 1:
    \begin{remark}\label{rem:initial_W}
        \begin{enumerate}[label=(\alph*)]
            \item If $x$ is a point around which $f$ is $C^2$, then the initial approximation in Step 1 is $W = \{ (x,f(x),\nabla f(x),\nabla^2 f(x)) \}$, so the first subproblem \eqref{eq:min_TxepsW} that is solved in Step 2 is the classic subproblem from the trust-region Newton method (cf.\ \cite{NW2006}, Chapter 4).
            \item For the proof of Theorem \ref{thm:algo2_termination}, it does not matter how $W$ is initialized. As such, elements of the current $\eps$-jet that were already sampled in previous iterations may be reused by including them in the initial $W$. (For example, by construction, the $\eps$-ball of each iterate contains the preceding iterate, so at least the information sampled at the preceding iterate can be reused.) This modification introduces bundle-like behavior into our approach and will be used for our numerical experiments in Section \ref{sec:numerical_experiments}.
        \end{enumerate}
    \end{remark}
    
\subsection{Practical algorithm} \label{subsec:practical_algorithm}

    By using Algo.\ \ref{algo:approx_jet} for approximating the $\eps$-jet in Algo.\ \ref{algo:abstract_descent_method}, we obtain Algo.\ \ref{algo:practical_descent_method}.
    \begin{algorithm} 
        \caption{Practical descent method}
        \label{algo:practical_descent_method}
   	    \begin{algorithmic}[1] 
            \Require Initial point $x^0 \in \R^n$, vanishing sequences $(\eps_j)_j$, $(\tau_j)_j \in \R^{>0}$, approximation parameter $c \in (0,1)$.
   	        \State Initialize $j = 1$, $i = 0$ and $x^{1,0} = x^0$.
   	        \State Compute $\theta^{j,i} = \theta^W(x^{j,i},\eps_j)$ and $\bar{z}^{j,i} = \bar{z}^W(x^{j,i},\eps)$ via Algo.\ \ref{algo:approx_jet}.
            \If{$(\theta^{j,i} - f(x^{j,i})) / \eps_j > -\tau_j$} \Comment{Check decrease}
                \State Set $x^{j+1,0} = x^{j,i}$, $j = j+1$ and $i = 0$. \Comment{Change tolerances}
            \Else
                \State Set $x^{j,i+1} = \bar{z}^{j,i}$ and $i = i+1$. \Comment{Descent step}
            \EndIf
            \State Go to Step 2.
        \end{algorithmic}
    \end{algorithm}	
    In contrast to Algo.\ \ref{algo:abstract_descent_method}, which required the entire $\eps$-jet at every $x \in \R^n$, Algo.\ \ref{algo:practical_descent_method} only requires a single, arbitrary element of $\J_0^2 f(x)$ at every $x \in \R^n$. Let $(N_j)_j$, $(x^j)_j$ and $(\hat{x}^l)_l$ be defined analogously to \eqref{eq:def_xj} and \eqref{eq:def_xhatl}. By our construction in Section \ref{subsec:approximating_eps_jet}, Step 6 in Algo.\ \ref{algo:practical_descent_method} can only be reached when \eqref{eq:fzW_decrease} holds. This means that $(f(x^j))_j$ and $(f(\hat{x}^l))_l$ are again non-increasing sequences. Furthermore, the following theorem shows that we obtain the same convergence result as for Algo.\ \ref{algo:abstract_descent_method}:
    \begin{theorem} \label{thm:algo3_convergence}
        Assume that $f$ satisfies Assumption \ref{assum:A3} or Assumption \ref{assum:A4}. Let $(x^j)_j$ and $(\hat{x}^j)_j$ be the sequences generated by Algo.\ \ref{algo:practical_descent_method}.
        If the level set $\{ y \in \R^n : f(y) \leq f(x^0) \}$ is bounded, then $(x^j)_j$ has an accumulation point $x^* \in \R^n$ and all accumulation points of $(x^j)_j$ are critical. Furthermore, $f(\hat{x}^l) \rightarrow f(x^*)$.
    \end{theorem}
    \vspace{-15pt}
    \begin{proof}
        Existence of an accumulation point $x^*$ and $f(\hat{x}^l) \rightarrow f(x^*)$ follows as in the proof of Theorem \ref{thm:algo1_convergence}.
        Due to \eqref{eq:fzW_decrease}, we have
        \begin{align*}
            f(x^{j,i+1}) - f(x^{j,i})
            \leq - c \eps_j \tau_j \quad \forall j \in \N, i \in \{0,\dots,N_j-1\}.
        \end{align*}
        Since $f$ is bounded below in $U$ (due to continuity), all $N_j$ have to be finite. This means that the condition in Step 3 has to hold infinitely many times, i.e.,
        \begin{align*}
            -\tau_j
            < \frac{\theta^W(x^{j,N_j},\eps_j) - f(x^{j,N_j})}{\eps_j}
            &= \frac{\theta^W(x^j,\eps_j) - f(x^j)}{\eps_j}
            \quad \forall j \in \N.
        \end{align*}
        Crucially, since $\theta^W \leq \theta$, this implies
        \begin{align*}
            -\tau_j 
            < \frac{\theta^W(x^j,\eps_j) - f(x^j)}{\eps_j}
            \leq \frac{\theta(x^j,\eps_j) - f(x^j)}{\eps_j}.
        \end{align*}
        The rest of the proof is identical to the proof of Theorem \ref{thm:algo1_convergence}. 
    \end{proof}

    We conclude this section with a discussion of the behavior of Algo.\ \ref{algo:practical_descent_method} when Assumptions \ref{assum:A3} and \ref{assum:A4} do not hold. For functions that merely satisfy Assumption \ref{assum:A1}, Algo.\ \ref{algo:practical_descent_method} still produces a non-increasing sequence, since Assumption \ref{assum:A1} is sufficient for Algo.\ \ref{algo:approx_jet} to terminate. However, the following simple example shows that even Assumption \ref{assum:A2} is not sufficient to obtain convergence:
    \begin{example} \label{example:semismooth}
        Consider the function $ f : \R^2 \rightarrow \R$,
        \begin{align*}
            x
            \mapsto \max(\min(f_1(x),f_2(x)),f_3(x))
            = \max(\min(x_1 - 2 x_2, x_1 + 2 x_2),x_2),
        \end{align*}
        which satisfies Assumption \ref{assum:A2} but not Assumption \ref{assum:A4}. The graph of $f$ is shown in Figure \ref{fig:example_semismooth}(a).
        \begin{figure}
            \centering
            \parbox[b]{0.49\textwidth}{
                \centering 
                \includegraphics[width=0.49\textwidth]{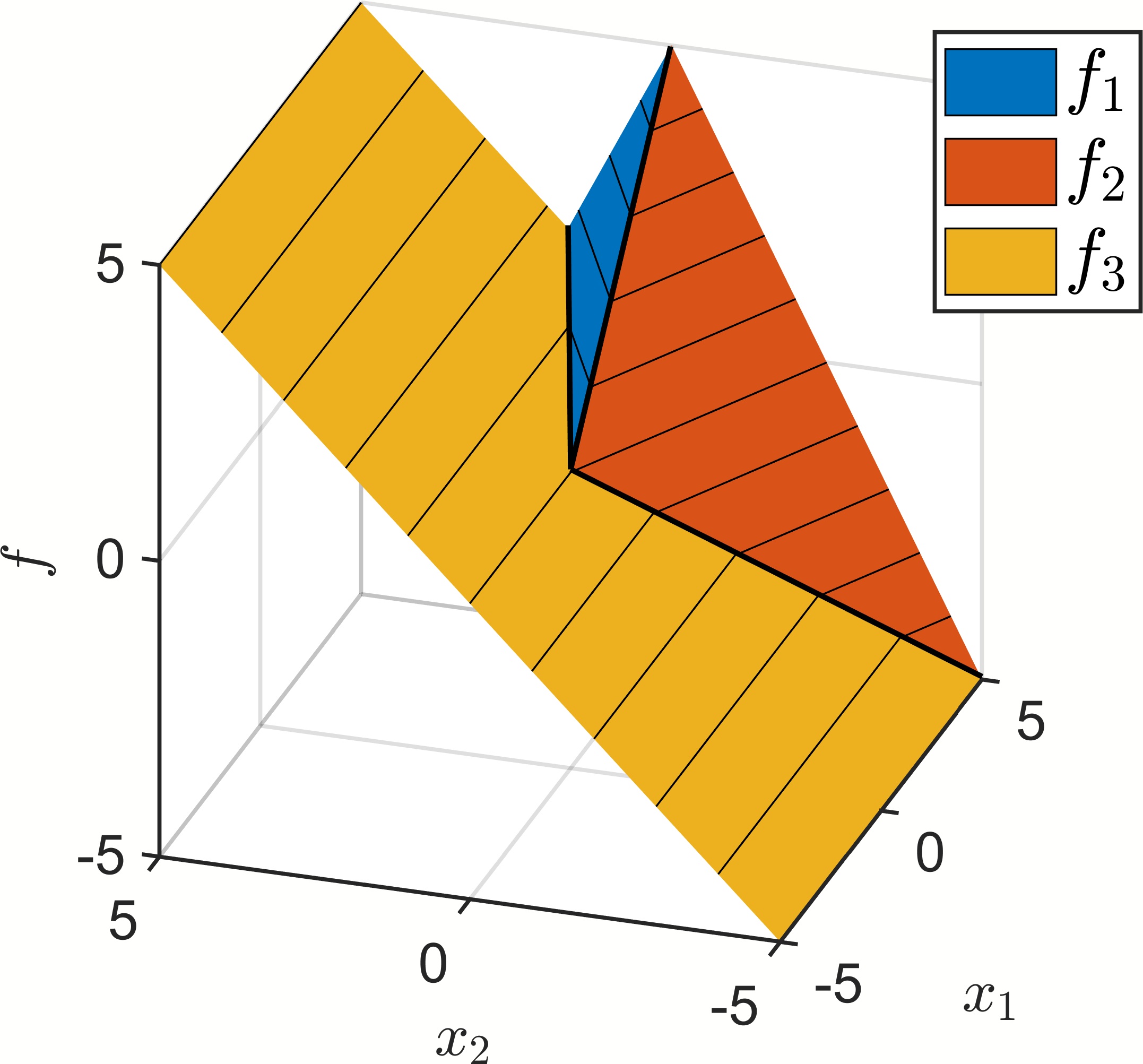}\\
                (a)
    		}
            \parbox[b]{0.49\textwidth}{
                \centering 
                \includegraphics[width=0.45\textwidth]{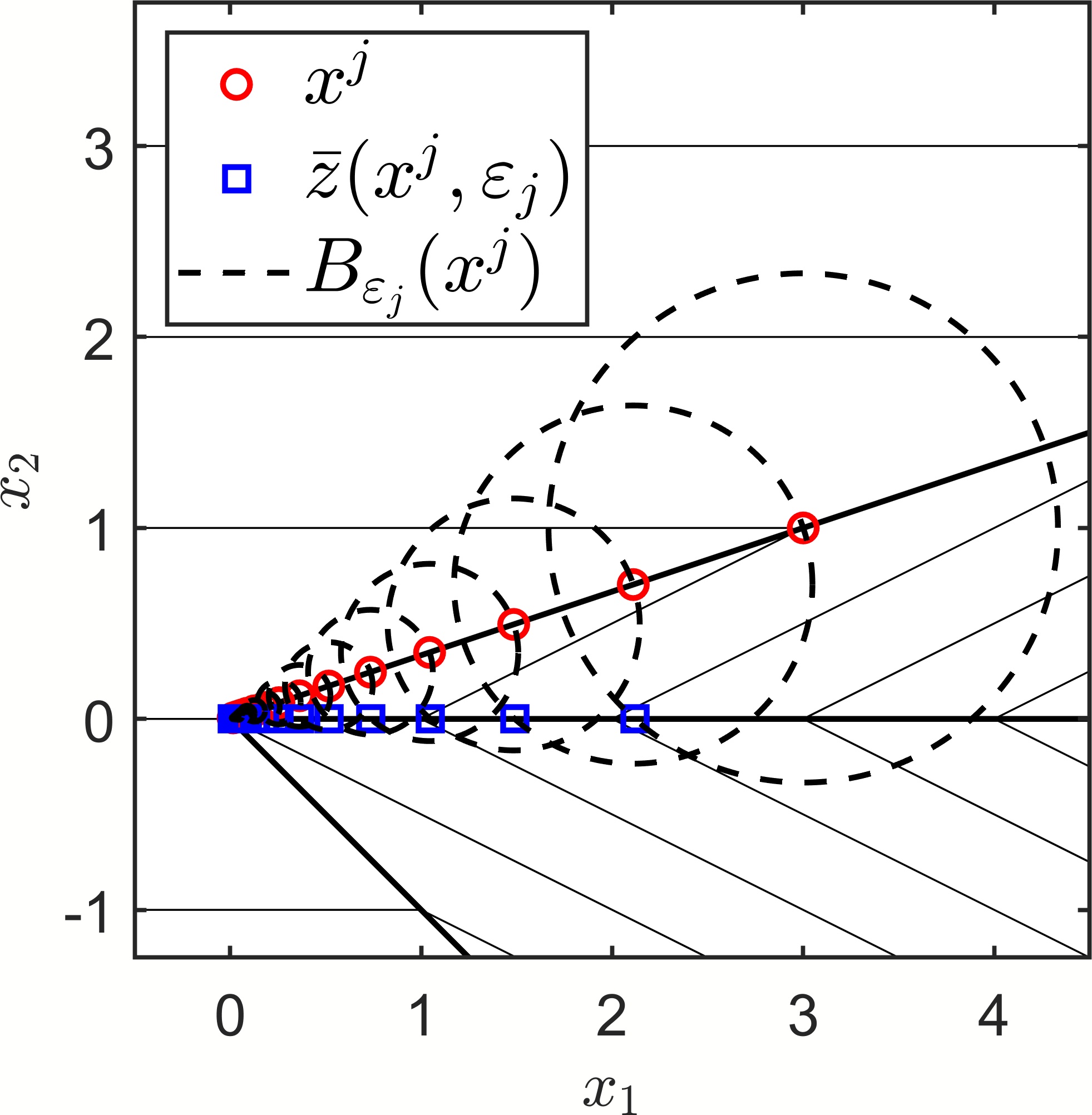}\\
                (b)
    		}
            \caption{(a) The graph of $f$ in Example \ref{example:semismooth}. (b) The sequences $(x^j)_j$ and $(\bar{z}(x^j,\eps_j))_j$ generated by Algo.\ \ref{algo:abstract_descent_method} and \ref{algo:practical_descent_method}.}
            \label{fig:example_semismooth}
        \end{figure}   
        Choose the parameters
        \begin{align*}
            x^0 = (3,1)^\top, \quad
            \eps_j = \frac{4}{3} \left( \frac{3}{3 + 4/\sqrt{10}} \right)^{j-1}, \quad
            \tau_j \in (0,1/\sqrt{10}) \quad \forall j \in \N.
        \end{align*}
        Then one can show that both Algo.\ \ref{algo:abstract_descent_method} and Algo.\ \ref{algo:practical_descent_method} generate the same sequence $(x^j)_j$ shown in Figure \ref{fig:example_semismooth}(b). (For Algo.\ \ref{algo:practical_descent_method}, this requires choosing the gradient of $f_3$ as a subgradient at points where both $f_1$ and $f_3$ are active.) Since $x^j \rightarrow 0 \in \R^2$ but $0$ is not a critical point of $f$, this shows that the convergence result does not apply here. The issue can be seen when considering the models $\T_{x^j,\eps_j}$: Since all selection functions are linear and active at some point in $B_{\eps_j}(x^j)$, it holds $\T_{x^j,\eps_j}(z) = \max(f_1(z),f_2(z),f_3(z))$. In Figure \ref{fig:example_semismooth}(b), it can be seen that the resulting model minimum $\bar{z}(x^j,\eps_j)$ does not yield a smaller objective value than $x^j$. Therefore, the algorithm increases $j$. With the smaller $\eps_{j+1}$ and $\tau_{j+1}$, it is able to perform a single descent step, but then immediately runs into the same issue as before by our choice of parameters. By construction of $f$, this behavior continues infinitely without the limit of $(x^j)_j$ being critical. In terms of the error bounds in Section \ref{subsec:second_order_model}, it is possible to show that
        \begin{align*}
            \max_{z \in B_{\eps_j}(x^j)} R_{x^j,\eps_j}(z) \geq \left( 3 - \frac{4}{\sqrt{10}} \right) \eps_j \quad \forall j \in \N,
        \end{align*}
        i.e., the modeling error is linear in $\eps$ instead of (at least) quadratic as in Lemma \ref{lem:error_convex} and Lemma \ref{lem:error_max_fun}.
    \end{example}

\section{Numerical experiments} \label{sec:numerical_experiments}

    In this section, we perform numerical experiments using a MATLAB implementation of Algo.\ \ref{algo:practical_descent_method}. The implementation is available at \url{https://github.com/b-gebken/SOGS}, including code for the reproduction of all results presented in this section. We will refer to this implementation as \sogs{} (\emph{second-order gradient sampling}) for the sake of brevity. (However, we want to emphasize that we do not claim that our method is the only way to incorporate second-order information into gradient sampling. We merely use this name for convenience.) In Section \ref{subsec:benchmark}, we compare our method to other solution methods for relatively general nonconvex, nonsmooth optimization problems. Afterwards, in Section \ref{subsec:superlinear}, we compare it to solvers that are able to achieve (provably) superlinear convergence for certain special classes of nonsmooth objective functions.

    In the following, we mention some details of our implementation: First of all, solutions of the subproblem \eqref{eq:min_TxepsW_quad} are computed via IPOPT \cite{WB2005} (with the Matlab interface from \cite{B2023}). (Note that this may only yield local solutions of \eqref{eq:min_TxepsW_quad}.) Since the original optimization problem is solved by consecutively solving the subproblem \eqref{eq:min_TxepsW_quad}, we can only solve the former problem up to the accuracy with which we solve the latter, where we simply use the default tolerances of IPOPT. Algo.\ \ref{algo:approx_jet} is initialized with all previously evaluated elements of the current $\eps$-jet, as discussed in Remark \ref{rem:initial_W}(b). The previously evaluated elements are stored in a queue (which could be regarded as a ``bundle'') with a maximum size of $100$. If the maximum size is reached and a new element is to be added, the oldest element is removed. For the parameters of Algo.\ \ref{algo:practical_descent_method}, we choose
    \begin{align*}
        c = 0.5, \quad \eps_j = (10 \cdot 0.1^{j-1})_j, \quad \tau_j = (10^{-5})_j,
    \end{align*}
    and we stop the algorithm as soon as $j = 6$ (after Step 4), i.e., the final pair of $(\eps_j,\tau_j)$ in Step 3 is $(10^{-3},10^{-5})$. While choosing $\tau_j$ as constant is not in the spirit of the convergence theory of our method, it turned out to be beneficial for the performance. (An analogous observation was made for the classical gradient sampling method, see \cite{BLO2005}, Section 4.)

    For the parameters of the other solvers we use in this section, we set the maximum number of iterations to $10^5$ and otherwise choose the default values that are given in the codes or the corresponding articles.

    \subsection{Performance on popular test problems} \label{subsec:benchmark}
    
        We compare the performance of \sogs{} to the performances of the following solution methods:
        \begin{itemize}
            \item \gradsamp{} \cite{BLO2005}\footnote{\url{https://cs.nyu.edu/~overton/papers/gradsamp/alg/}}: Classic gradient sampling method
            \item \dgs{} \cite{GP2021,G2024b,G2024a}\footnote{\url{https://github.com/b-gebken/DGS}}: Deterministic gradient sampling method
            \item \hanso{} \cite{LO2013}\footnote{\url{https://cs.nyu.edu/~overton/software/hanso/}}: Quasi-Newton (BFGS) method
            \item \slqpgs{} \cite{CO2012}\footnote{\url{https://github.com/frankecurtis/SLQPGS}}: SQP-method combined with gradient sampling
            \item \lmbm{} \cite{HMM2004,H2004}\footnote{\url{https://napsu.karmitsa.fi/lmbm/}}: Limited memory bundle method 
        \end{itemize}
        For \sogs{}, we have to be able to compute an element of the second-order $0$-jet at every $x \in \R^n$ (i.e., objective values, gradients and Hessians, cf.\ Lemma \ref{lem:jet_projections} and Lemma \ref{lem:jet_for_max_fun}). All other methods merely require objective values and subgradients. As test problems, we choose the 20 classical problems from \cite{H2004}, Appendix A, consisting of convex and nonconvex problems with prescribed initial points. They are scalable in the number of variables $n$, and we choose $n = 50$ for all problems. For computing derivatives, we use the exact analytic formulas.
                
        The performance metric we consider is the number of $\partial f$ evaluations. For \sogs{}, the number of $\partial f$ evaluations coincides with the number of $\nabla^2 f$ evaluations and is by one smaller than the number of $f$ evaluations. For \hanso{} and \lmbm{}, the number of $\partial f$ and $f$ evaluations coincide. For \gradsamp{} and \dgs{}, the number of $f$ evaluations is larger than the number of $\partial f$ evaluations. (Although it appears that the $f$ evaluations during the random sampling step in \gradsamp{} are unnecessary.) For \slqpgs{}, the number of $f$ evaluations is significantly lower than the number of $\partial f$ evaluations.

        \setlength\tabcolsep{0.5pt}
        \begin{table}
    		\centering
    		\caption{Results of applying the solvers listed in Section \ref{subsec:benchmark} to the 20 test problems in \cite{H2004}. For each solver, the left column shows the total number of $\partial f$ evaluations and the right column shows the distance of the smallest found objective value to the optimal value. (For test problems where the exact optimal value is unknown, we used the smallest objective values that we encountered during all of our experiments. They can be found alongside our implementation.)}
    		\label{table:performance}
    			\begin{tabular}{|c !{\vrule width 1.5pt} c | c !{\vrule width 1.5pt} c | c !{\vrule width 1.5pt} c | c !{\vrule width 1.5pt} c | c !{\vrule width 1.5pt} c | c !{\vrule width 1.5pt} c | c !{\vrule width 1.5pt}}
                    \hline
    				& \multicolumn{2}{c !{\vrule width 1.5pt}}{\sogs{}} & \multicolumn{2}{c !{\vrule width 1.5pt}}{\gradsamp{}} & \multicolumn{2}{c !{\vrule width 1.5pt}}{\dgs{}} & \multicolumn{2}{c !{\vrule width 1.5pt}}{\hanso{}} & \multicolumn{2}{c !{\vrule width 1.5pt}}{\slqpgs{}} & \multicolumn{2}{c !{\vrule width 1.5pt}}{\lmbm{}} \\
    				\hline
    				No. & \#$\partial f$ & Acc. & \#$\partial f$ & Acc. & \#$\partial f$ & Acc. & \#$\partial f$ & Acc. & \#$\partial f$ & Acc. & \#$\partial f$ & Acc. \\
    				\hline
                    1. & 378 & 3.0e-15 & 61200 & 2.8e-13 & 687 & 1.3e-06 & 968 & 2.2e-08 & 46864 & 2.6e-06 & 569 & 4.8e-07 \\ 
                    \hline 
                    2. & 107 & 8.3e-06 & 14400 & 2.5e-07 & 163 & 1.1e-04 & 307 & 3.9e-08 & 11009 & 8.6e-03 & 674 & 2.4e-05 \\ 
                    \hline 
                    3. & 71 & 7.3e-08 & 38900 & 8.3e-06 & 727 & 1.0e-05 & 574 & 2.2e-06 & 27876 & 5.3e-06 & 583 & 3.6e-08 \\ 
                    \hline 
                    4. & 545 & 1.2e-07 & 24300 & 5.2e-05 & 284 & 1.4e-04 & 968 & 4.5e-05 & 25553 & 1.9e-05 & 239 & 4.8e-07 \\ 
                    \hline 
                    5. & 26 & 9.7e-09 & 115400 & 1.6e-05 & 94 & 4.4e-05 & 129 & 2.2e-09 & 4343 & 2.9e-06 & 180 & 2.7e-09 \\ 
                    \hline 
                    6. & 22 & 1.4e-08 & 6100 & 5.3e-07 & 22 & 4.1e-07 & 27 & 2.4e-05 & 29593 & 1.6e-06 & 156 & 3.2e-08 \\ 
                    \hline 
                    7. & 250 & 8.7e-08 & 9500 & 2.5e-05 & 104 & 5.8e-05 & 102 & 1.2e-04 & 16261 & 8.9e-06 & 633 & 5.4e-09 \\ 
                    \hline 
                    8. & 436 & 1.7e-07 & 55500 & 4.5e-05 & 1564 & 2.7e-05 & 1136 & 6.6e-06 & 28684 & 4.8e-05 & 2765 & 6.9e-04 \\ 
                    \hline 
                    9. & 12 & 1.7e-09 & 37800 & 1.3e-06 & 60 & 4.4e-05 & 51 & 9.4e-06 & 5050 & 5.0e-06 & 112 & 1.5e-09 \\ 
                    \hline 
                    10. & 14 & 6.0e-09 & 61900 & 2.0 & 331 & 2.9e-05 & 453 & 7.4e-07 & 73124 & 8.7e-06 & 873 & 3.7e-07 \\ 
                    \hline 
                    11. & 452 & 3.2e-10 & 61100 & 5.0e-07 & 814 & 1.0e-06 & 100 & 0.0 & 850824 & 6.6e-07 & 10002 & 4.5e+01 \\ 
                    \hline 
                    12. & 178 & 2.6e-05 & 13500 & 2.9e-06 & 874 & 2.1e-05 & 250 & 2.7e-08 & 10807 & 3.2e-04 & 1026 & 7.0e-07 \\ 
                    \hline 
                    13. & 129 & 2.0e-10 & 42600 & 2.8e-06 & 650 & 6.6e-06 & 32 & 3.0 & 67165 & 1.1e-05 & 49 & 2.0 \\ 
                    \hline 
                    14. & 537 & 5.3e+02 & 70500 & 5.3e+02 & 1267 & 5.3e+02 & 675 & 5.3e+02 & 94536 & 5.3e+02 & 473 & 5.3e+02 \\ 
                    \hline 
                    15. & 595 & 9.2e-07 & 65900 & 1.9e-04 & 5686 & 5.9e-04 & 401 & 2.1e-04 & 11009 & 8.0e-05 & 815 & 4.4e-08 \\ 
                    \hline 
                    16. & 496 & 6.4e-08 & 35400 & 5.0e-07 & 436 & 1.1e-06 & 32 & 2.1e-02 & 9797 & 3.5e-07 & 1618 & 2.1e-05 \\ 
                    \hline 
                    17. & 1202 & 1.9e-11 & 39100 & 1.5e-11 & 657 & 1.6e-07 & 1952 & 1.3e-09 & 13534 & 1.7e-05 & 250 & 1.0 \\ 
                    \hline 
                    18. & 109 & 2.6e-09 & 232600 & 3.1e-06 & 1939 & 5.8e-06 & 2012 & 1.3e-04 & 23735 & 2.9e-06 & 190 & 5.0e-01 \\ 
                    \hline 
                    19. & 292 & 9.3e-08 & 269700 & 4.3e-04 & 5982 & 2.7e-04 & 16937 & 1.2e-08 & 5252 & 5.6e-04 & 57 & 5.8e-04 \\ 
                    \hline 
                    20. & 2041 & 8.0e-08 & 634400 & 3.4e-02 & 5036 & 4.1e-06 & 539 & 4.2e+01 & 33734 & 5.8e-03 & 1137 & 6.7e-02 \\ 
                    \hline 
    			\end{tabular}	
    	\end{table}
        
        The results are shown in Table \ref{table:performance}. For \gradsamp{} and \slqpgs{}, we see that they need significantly more $\partial f$ evaluations than the other methods. This is no surprise, as both methods employ random sampling, which is known to be an inefficient way (with respect to the number of evaluations) to approximate the $\eps$-subdifferential. For the remaining methods, we see that \sogs{} has both the highest accuracy and the smallest number of $\partial f$ evaluations on Problems 1, 6, 8, 10 and 18, and for \hanso{}, this is true for Problem 11. For all other problems, Table \ref{table:performance} cannot be used to determine a favorable method, as the highest accuracy is not achieved with the smallest number of evaluations.
        
        In theory, a proper comparison is only possible if the parameters for all methods are chosen in a way that results of the same accuracy are computed. Unfortunately, due to the diversity of the stopping criteria employed in these methods, we were unable to achieve this. (Furthermore, while all methods are iterative methods, only \sogs{}, \dgs{} and \slqpgs{} return the entire sequence of iterates.) Nonetheless, in an attempt to emulate a perfect tuning of parameters, we consider a second performance metric: For all methods except \sogs{}, we record all points in which a subgradient was evaluated. (By construction of each method, the actual iterates are a subset of these points.) After each method finishes, we evaluate the objective function in each of these points, and count how many $\partial f$ evaluations were required to encounter an objective value whose distance to the best value is at most $10^{-4}$. Clearly, this performance metric is not a fair comparison, as it ignores the additional evaluations the methods would have to perform to ``naturally'' terminate. (This may even lead to smaller objective values than the ones identified by the methods, as they may not factor in objective values at points where the subgradient is evaluated.) To avoid giving any unfair advantage to our own method, we only consider the actual iterates $x^{j,i}$ for \sogs{}. Additionally, once an iterate with a sufficiently small objective value is encountered, we also count the evaluations during the sampling step at that iterate, as these would be required for our method to naturally terminate. The results are shown in Table \ref{table:performance_cutoff}. We see that \sogs{} has the least number of $\partial f$ evaluations for 13 out of the 20 test problems (and for Problem 14, none of the methods found the minimum).

        \begin{table}
    		\centering
    		\caption{Same as Table \ref{table:performance}, but all methods are artificially terminated once they encounter the first point with an accuracy of at least $10^{-4}$. For every problem, the data for the method with the least number of $\partial f$ evaluations is written in bold. A dash indicates that a method did not achieve the accuracy threshold at all.} 
    		\label{table:performance_cutoff}
    			\begin{tabular}{|c !{\vrule width 1.5pt} c | c !{\vrule width 1.5pt} c | c !{\vrule width 1.5pt} c | c !{\vrule width 1.5pt} c | c !{\vrule width 1.5pt} c | c !{\vrule width 1.5pt} c | c !{\vrule width 1.5pt}}
                    \hline
    				& \multicolumn{2}{c !{\vrule width 1.5pt}}{\sogs{}} & \multicolumn{2}{c !{\vrule width 1.5pt}}{\gradsamp{}} & \multicolumn{2}{c !{\vrule width 1.5pt}}{\dgs{}} & \multicolumn{2}{c !{\vrule width 1.5pt}}{\hanso{}} & \multicolumn{2}{c !{\vrule width 1.5pt}}{\slqpgs{}} & \multicolumn{2}{c !{\vrule width 1.5pt}}{\lmbm{}} \\
    				\hline
    				No. & \#$\partial f$ & Acc. & \#$\partial f$ & Acc. & \#$\partial f$ & Acc. & \#$\partial f$ & Acc. & \#$\partial f$ & Acc. & \#$\partial f$ & Acc. \\
    				\hline
                    1. & \textbf{374} & \textbf{3.0e-15} & 51701 & 6.8e-05 & 573 & 9.7e-05 & 451 & 9.1e-05 & 35250 & 9.8e-05 & 446 & 9.5e-05 \\ 
                    \hline 
                    2. & \textbf{22} & \textbf{3.8e-05} & 4709 & 8.7e-05 & - & - & 105 & 8.5e-05 & - & - & 316 & 4.7e-05 \\ 
                    \hline 
                    3. & \textbf{27} & \textbf{4.4e-05} & 24101 & 7.4e-05 & 429 & 9.4e-05 & 428 & 9.5e-05 & 13939 & 9.9e-05 & 403 & 8.2e-05 \\ 
                    \hline 
                    4. & 482 & 1.9e-05 & 22001 & 9.6e-05 & 196 & 9.3e-05 & 879 & 9.9e-05 & 17474 & 8.7e-05 & \textbf{103} & \textbf{9.7e-05} \\ 
                    \hline 
                    5. & \textbf{21} & \textbf{2.3e-07} & 17391 & 5.6e-05 & 88 & 6.5e-05 & 73 & 3.9e-05 & 2631 & 8.8e-05 & 80 & 6.1e-05 \\ 
                    \hline 
                    6. & \textbf{15} & \textbf{3.1e-05} & 3102 & 9.7e-05 & 18 & 4.3e-05 & 27 & 2.4e-05 & 28042 & 7.9e-05 & 74 & 1.4e-05 \\ 
                    \hline 
                    7. & 170 & 6.0e-05 & 6001 & 7.7e-05 & 99 & 4.4e-05 & 100 & 1.7e-05 & 10000 & 7.6e-05 & \textbf{83} & \textbf{9.6e-05} \\ 
                    \hline 
                    8. & \textbf{73} & \textbf{8.2e-05} & 43601 & 8.6e-05 & 817 & 9.9e-05 & 850 & 9.7e-05 & 20403 & 9.9e-05 & - & - \\ 
                    \hline 
                    9. & \textbf{4} & \textbf{1.7e-09} & 30815 & 1.4e-05 & 59 & 2.1e-05 & 41 & 3.2e-05 & 2991 & 7.2e-05 & 37 & 9.9e-05 \\ 
                    \hline 
                    10. & \textbf{6} & \textbf{6.0e-09} & - & - & 298 & 8.1e-05 & 237 & 9.6e-05 & 39795 & 9.7e-05 & 207 & 8.1e-05 \\ 
                    \hline 
                    11. & 444 & 3.2e-10 & 55001 & 9.7e-05 & 639 & 9.6e-05 & \textbf{100} & \textbf{0.0} & 847707 & 9.9e-05 & - & - \\ 
                    \hline 
                    12. & \textbf{14} & \textbf{4.8e-05} & 6810 & 9.3e-05 & 497 & 8.1e-05 & 94 & 9.7e-05 & - & - & 483 & 7.9e-05 \\ 
                    \hline 
                    13. & \textbf{122} & \textbf{2.0e-10} & 37901 & 7.2e-05 & 535 & 9.9e-05 & - & - & 52622 & 9.9e-05 & - & - \\ 
                    \hline 
                    14. & - & - & - & - & - & - & - & - & - & - & - & - \\ 
                    \hline 
                    15. & 522 & 4.0e-05 & 65601 & 6.3e-05 & - & - & - & - & 10682 & 8.6e-05 & \textbf{397} & \textbf{7.8e-05} \\ 
                    \hline 
                    16. & 274 & 2.4e-05 & 30101 & 7.3e-05 & \textbf{257} & \textbf{9.7e-05} & - & - & 5556 & 9.4e-05 & 1519 & 9.4e-05 \\ 
                    \hline 
                    17. & 835 & 8.1e-05 & 15601 & 9.7e-05 & \textbf{330} & \textbf{8.1e-05} & 482 & 9.9e-05 & 9192 & 9.8e-05 & - & - \\ 
                    \hline 
                    18. & \textbf{101} & \textbf{2.6e-09} & 168001 & 9.8e-05 & 1496 & 1.0e-04 & - & - & 18686 & 9.1e-05 & - & - \\ 
                    \hline 
                    19. & \textbf{109} & \textbf{5.7e-05} & - & - & - & - & 10646 & 9.9e-05 & - & - & - & - \\ 
                    \hline 
                    20. & \textbf{1948} & \textbf{9.6e-05} & - & - & 3345 & 9.9e-05 & - & - & - & - & - & - \\ 
                    \hline 
    			\end{tabular}	
    	\end{table}
                
        While these results suggest that \sogs{} is more efficient than the other methods in terms of overall oracle calls, we emphasize that \sogs{} is also the only method that requires second-order information. Additionally, even when ignoring $\nabla^2 f$ evaluations, the computations that have to be performed in every iteration of \sogs{} (i.e., the solution of subproblem \eqref{eq:min_TxepsW_quad}) are significantly more expensive than the ones in other methods. This can be seen in Table \ref{table:runtimes}, which compares the runtime of \sogs{} to the runtimes of the other solvers for our set of test problems. As such, it appears that \sogs{} in its current state is only useful when Hessian matrices are available and oracle calls are relatively expensive. 

        \begin{table}
    		\centering
    		\caption{Comparison of the runtime (in seconds) of \sogs{} to the runtimes of the other solvers. Each column corresponds to one other solver, and the times are the summed up runtimes on all problems where both \sogs{} and the other solver got to within $10^{-4}$ of the exact optimal value.} 
    		\label{table:runtimes}
                \begin{tabular}{| c | c | c | c | c | c |}
                    \hline
                     & vs.\ \gradsamp{} & vs.\ \dgs{} & vs.\ \hanso{} & vs.\ \slqpgs{} & vs.\ \lmbm{} \\
                    \hline
                    \sogs{} & 269.3 & 270.4 & 166.4 & 260.4 & 176.0 \\ 
                    \hline 
                    Other solver & 176.5 & 16.1 & 8.3 & 210.2 & 5.6 \\ 
                    \hline 
    			\end{tabular}	
    	\end{table}

    \subsection{Comparison to superlinear solvers for special problem classes} \label{subsec:superlinear}

        In the previous experiment, we examined the performance of \sogs{} in terms of total oracle calls and runtime. In the following, we attempt to gain insight into the more theoretical property of its rate of convergence. To this end, we compare \sogs{} to the two methods that achieve superlinear convergence on two classes of nonsmooth problems with special structure, which are the \vu-algorithm \cite{MS2005} and SuperPolyak \cite{CD2024} (cf.\ Section \ref{sec:introduction}).

        For our numerical experiments with the \vu-algorithm, we use the implementation \vubundle{} from the Julia package \texttt{NonSmoothSolvers.jl}\footnote{\url{https://github.com/GillesBareilles/NonSmoothSolvers.jl}}. As a test problem, we use the \emph{half-and-half} function from \cite{LO2008} (also considered in \cite{MS2012}), where
        \begin{equation} \label{eq:halfhalf}
            \begin{aligned}
                &f : \R^8 \rightarrow \R, \quad x \mapsto \sqrt{x^\top A x} + x^\top B x, \\
                &A_{i,j} := 
                \begin{cases}
                    1, & i = j \in \{1,3,5,7\},\\
                    0, & \text{otherwise,}
                \end{cases}, \quad 
                B_{i,j} := 
                \begin{cases}
                    1/i^2, & i = j,\\
                    0, & \text{otherwise.}
                \end{cases}
            \end{aligned}
        \end{equation}
        
        As in \cite{MS2012} the initial point $x^0 = (20.08,\dots,20.08)^\top \in \R^8$ is used. Since the default accuracy of \vubundle{} is higher than the default accuracy of IPOPT, we change the latter from the default $10^{-8}$ to $10^{-10}$ to compute results of similar quality. (See the parameter \verb|tol| in the documentation of IPOPT for details.) The remaining parameters of \sogs{} are the same as stated at the beginning of this section. The result is shown in Figure \ref{fig:example_superlinear}(a).
        \begin{figure}
            \centering
            \parbox[b]{0.49\textwidth}{
                \centering 
                \includegraphics[width=0.45\textwidth]{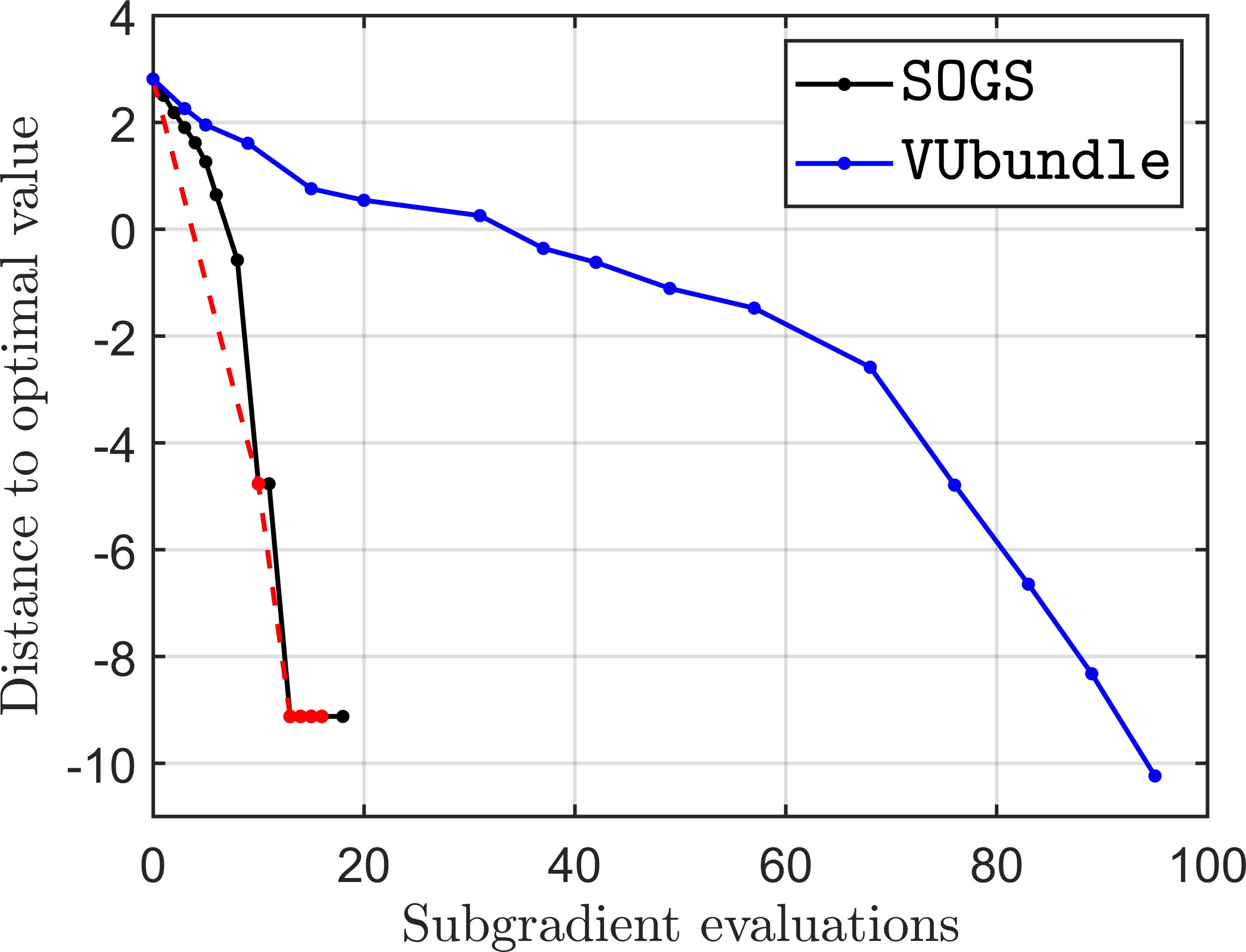}\\
                (a)
    		}
            \parbox[b]{0.49\textwidth}{
                \centering 
                \includegraphics[width=0.45\textwidth]{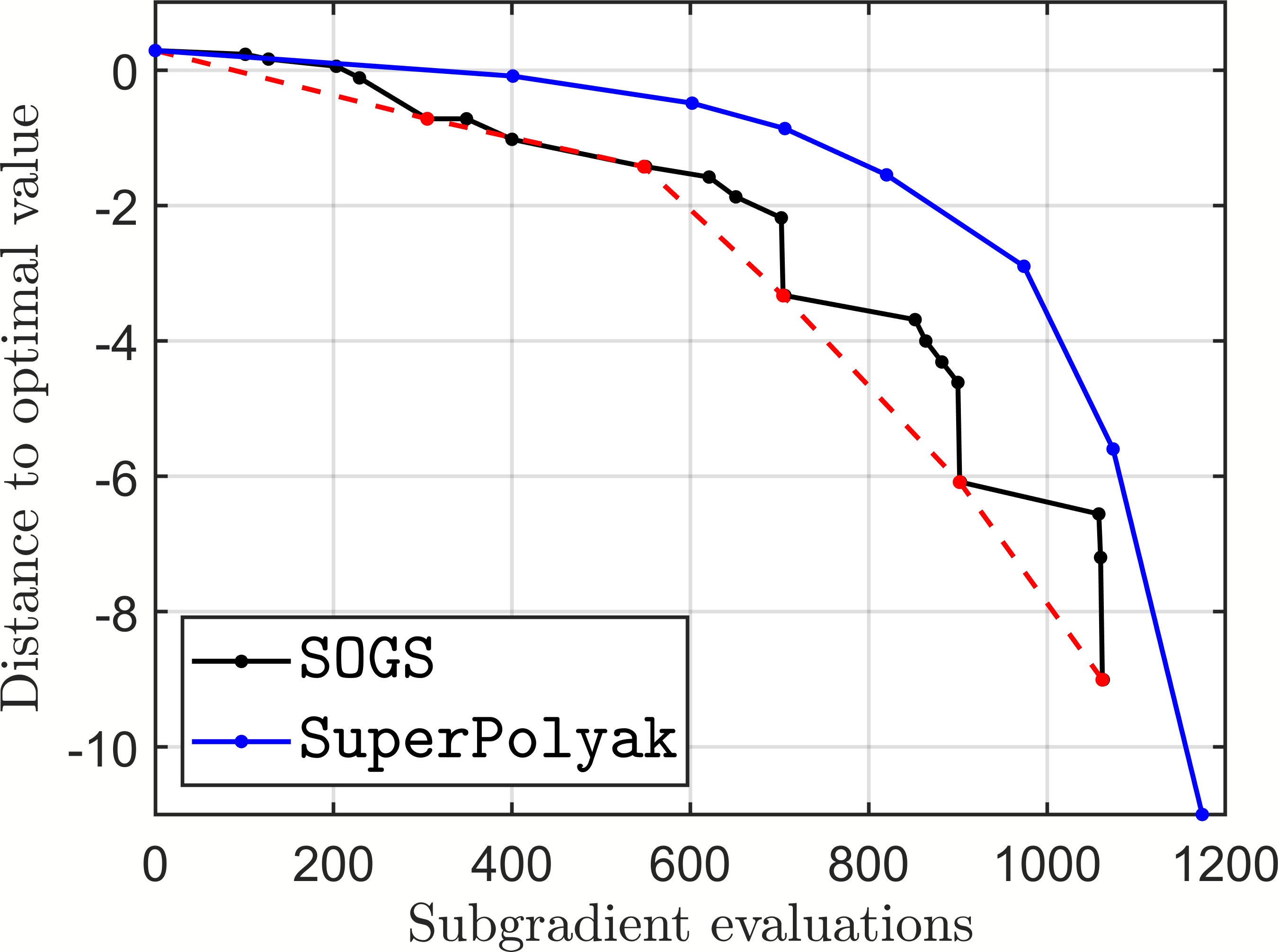}\\
                (b)
    		}
            \caption{(a) The dots represent the iterates of \sogs{} (i.e., $(\hat{x}^l)_l$) and \vubundle{} for Problem \eqref{eq:halfhalf}. The horizontal axis shows the number of $\partial f$ evaluations required up to each iterate and the vertical axis shows the distance (in logarithmic scale) of the objective value to the minimal value at each iterate. For \sogs{}, the red dots highlight the subsequence $(x^j)_j$ among $(\hat{x}^l)_l$, cf.\ \eqref{eq:def_xj}, \eqref{eq:def_xhatl}. (b) Same as (a) for the solver \superpolyak{} and Problem \eqref{eq:max_root}. (Not shown is the final iterate of \superpolyak{}, which took 1274 subgradient evaluations and reached the optimal value up to machine precision.)}
            \label{fig:example_superlinear}
        \end{figure}   
        For \vubundle{} we observe (roughly) superlinear convergence, as expected. For the first iterates of \sogs{} we observe the same. However, at an accuracy of about $10^{-9}$, \sogs{} gets stuck. This is not a surprise, as this is around the accuracy of IPOPT. (Unfortunately, we were unable to further increase the accuracy of IPOPT by further lowering the tolerances.) When comparing the two methods, we see that \sogs{} requires fewer overall $\partial f$ evaluations than \vubundle{}.

        For SuperPolyak we use the Julia implementation \texttt{SuperPolyak.jl}\footnote{\url{https://github.com/COR-OPT/SuperPolyak.jl}}. As a test problem, we consider the nonconvex function
        \begin{align} \label{eq:max_root}
            f_a : \R^n \rightarrow \R, \quad x \mapsto \max_{i \in \{1,\dots,n\}} \left( \sqrt{|x_i| + a} - \sqrt{a} \right)
        \end{align}
        for $a = 0.1$ and $n = 100$ and the initial point $x^0 = (5,\dots,5)^\top \in \R^n$. (Note that this function is locally Lipschitz for $a > 0$.) It is easy to see that $x^* = 0 \in \R^n$ is a sharp minimum of $f$. For the same reason as above, we set the accuracy of IPOPT to $10^{-10}$ and leave all other parameters unchanged. The result is shown in Figure \ref{fig:example_superlinear}(b). For \superpolyak{} we observe (remarkably clean) superlinear convergence, as expected. For \sogs{} there is clearly no superlinear convergence in the sense of Q-convergence (cf.\ \cite{NW2006}, Appendix A.2). However, there still seems to be superlinear R-convergence (i.e., there seems to be some $(\nu_l)_l \in \R^{>0}$ that vanishes Q-superlinearly with $| f(\hat{x}^l) - f(x^*) | \leq \nu_l$ for all $l \in \N$). Furthermore, the subsequence $(x^j)_j$ appears to converge Q-superlinearly with a bounded number of subgradient evaluations in between each iterate.

\section{Conclusion} \label{sec:conclusion}

    In this article, we introduced a new concept for second-order information of nonsmooth functions, the second-order $\eps$-jet, and used it to construct a second-order gradient sampling method. We showed convergence of this method for the case where the objective is convex or of max-type. While we did not provide any theoretical results on its speed of convergence, our numerical experiments suggest that in terms of oracle calls, it is fast (in a sense we have yet to theoretically capture).

    We believe that there are many possibilities for future work building on this article:
    \begin{itemize}
        \item While our approach appears to be efficient in terms of oracle calls, it is not efficient in terms of computational cost per iteration. The largest bottleneck is the solution of the subproblem \eqref{eq:min_TxepsW_quad}. In every iteration of the practical Algo.\ \ref{algo:practical_descent_method}, it has to be solved once for computing the next iterate, but potentially many times prior to that for approximating the $\eps$-jet via Algo.\ \ref{algo:approx_jet}. To make things worse, we technically require global solutions of this subproblem. (However, considering our numerical experiments, local solutions do not seem to cause any issues.) In \sogs{} this subproblem is simply treated and solved as a general nonlinear, constrained optimization method. But considering its special structure as a QCQP, there may be more efficient ways of dealing with it. (For example, using a method like \cite{L2005}.) Furthermore, since two consecutive subproblems solved in Algo.\ \ref{algo:approx_jet} only differ by one constraint, it may be possible to warm-start the solution process in some way. Alternatively, inexact solutions of the subproblem could be sufficient. 
        \item A second issue of our approach is the need for Hessian matrices. An obvious question is whether quasi-Newton strategies can be employed instead. Since we essentially consider multiple Hessian matrices simultaneously in every iteration, this motivates the use of multiple quasi-Newton matrices for their approximation. While quasi-Newton ideas have appeared before in nonsmooth optimization, the idea of using multiple approximating matrices in this way is, to the best of our knowledge, novel.
        \item We believe that the behavior in Figure \ref{fig:example_superlinear}(b), i.e., superlinear R-convergence of $(\hat{x}^l)_l$, superlinear Q-convergence of $(x^j)_j$ and a bounded number of subgradient evaluations in between, could be the general behavior of our method for a suitable class of objective functions and proper choices of $(\eps_j)_j$ and $(\tau_j)_j$. An actual proof might be possible by exploiting the cubic error estimate in Lemma \ref{lem:error_max_fun}. (For the results in this article, we only used the fact that this error is quadratic, but not that it is actually cubic.) Furthermore, in \cite{G2024a}, it was shown how the speed of convergence of an arbitrary sequence $(x^j)_j \in \R^n$ to a critical point can be inferred from the speed at which $\min(\| \partial_{\eps_j} f(x^j) \|)$ vanishes. In our theory, $\min(\| \partial_{\eps_j} f(x^j) \|)$ appears in Lemma \ref{lem:model_descent_estimate}(b), by which it is related to $(\tau_j)_j$ via the condition in Step 3 of Algo.\ \ref{algo:practical_descent_method}.
        \item As in trust-region methods, it could turn out to be beneficial to include a mechanism for increasing $\eps$ into our approach.
        \item Aside from the method in this article, there may be other uses for the second-order $\eps$-jet (or $0$-jet) in nonsmooth optimization. For example, it might fit as the theoretical foundation for a proper analysis of smooth quasi-Newton methods in the nonsmooth setting (cf.\ the challenge in Section 7 in \cite{LO2013}).
        \item It would be interesting to analyze the relationship of $\Jeps(x)$ to the second-order theory in convex analysis (cf.\ \cite{RW1998}, Chapter 13). For example, by Alexandrov's theorem (cf.\ \cite{NP2006}, Theorem 3.11.2), convexity implies the existence of a so-called \emph{Alexandrov Hessian} almost everywhere, which is likely related to the second-order information in $\Jeps(x)$.
    \end{itemize}

\backmatter

\noindent \textbf{Acknowledgements.} \quad This research was funded by Deutsche Forschungsgemeinschaft (DFG, German Research Foundation) via the projects 545166481 (WBP Position) and 314151124 (Priority Programme 1962).

\bibliography{references}


\begin{thebibliography}{51}
\ifx \bisbn   \undefined \def \bisbn  #1{ISBN #1}\fi
\ifx \binits  \undefined \def \binits#1{#1}\fi
\ifx \bauthor  \undefined \def \bauthor#1{#1}\fi
\ifx \batitle  \undefined \def \batitle#1{#1}\fi
\ifx \bjtitle  \undefined \def \bjtitle#1{#1}\fi
\ifx \bvolume  \undefined \def \bvolume#1{\textbf{#1}}\fi
\ifx \byear  \undefined \def \byear#1{#1}\fi
\ifx \bissue  \undefined \def \bissue#1{#1}\fi
\ifx \bfpage  \undefined \def \bfpage#1{#1}\fi
\ifx \blpage  \undefined \def \blpage #1{#1}\fi
\ifx \burl  \undefined \def \burl#1{\textsf{#1}}\fi
\ifx \doiurl  \undefined \def \doiurl#1{\url{https://doi.org/#1}}\fi
\ifx \betal  \undefined \def \betal{\textit{et al.}}\fi
\ifx \binstitute  \undefined \def \binstitute#1{#1}\fi
\ifx \binstitutionaled  \undefined \def \binstitutionaled#1{#1}\fi
\ifx \bctitle  \undefined \def \bctitle#1{#1}\fi
\ifx \beditor  \undefined \def \beditor#1{#1}\fi
\ifx \bpublisher  \undefined \def \bpublisher#1{#1}\fi
\ifx \bbtitle  \undefined \def \bbtitle#1{#1}\fi
\ifx \bedition  \undefined \def \bedition#1{#1}\fi
\ifx \bseriesno  \undefined \def \bseriesno#1{#1}\fi
\ifx \blocation  \undefined \def \blocation#1{#1}\fi
\ifx \bsertitle  \undefined \def \bsertitle#1{#1}\fi
\ifx \bsnm \undefined \def \bsnm#1{#1}\fi
\ifx \bsuffix \undefined \def \bsuffix#1{#1}\fi
\ifx \bparticle \undefined \def \bparticle#1{#1}\fi
\ifx \barticle \undefined \def \barticle#1{#1}\fi
\bibcommenthead
\ifx \bconfdate \undefined \def \bconfdate #1{#1}\fi
\ifx \botherref \undefined \def \botherref #1{#1}\fi
\ifx \url \undefined \def \url#1{\textsf{#1}}\fi
\ifx \bchapter \undefined \def \bchapter#1{#1}\fi
\ifx \bbook \undefined \def \bbook#1{#1}\fi
\ifx \bcomment \undefined \def \bcomment#1{#1}\fi
\ifx \oauthor \undefined \def \oauthor#1{#1}\fi
\ifx \citeauthoryear \undefined \def \citeauthoryear#1{#1}\fi
\ifx \endbibitem  \undefined \def \endbibitem {}\fi
\ifx \bconflocation  \undefined \def \bconflocation#1{#1}\fi
\ifx \arxivurl  \undefined \def \arxivurl#1{\textsf{#1}}\fi
\csname PreBibitemsHook\endcsname

\bibitem[\protect\citeauthoryear{Mäkelä and Neittaanmäki}{1992}]{MN1992}
\begin{bbook}
\bauthor{\bsnm{Mäkelä}, \binits{M.M.}},
\bauthor{\bsnm{Neittaanmäki}, \binits{P.}}:
\bbtitle{Nonsmooth Optimization}.
\bpublisher{World Scientific},
\blocation{Singapore}
(\byear{1992}).
\doiurl{10.1142/1493}
\end{bbook}
\endbibitem

\bibitem[\protect\citeauthoryear{Bagirov et~al.}{2014}]{BKM2014}
\begin{bbook}
\bauthor{\bsnm{Bagirov}, \binits{A.}},
\bauthor{\bsnm{Karmitsa}, \binits{N.}},
\bauthor{\bsnm{M\"{a}kel\"{a}}, \binits{M.M.}}:
\bbtitle{Introduction to Nonsmooth Optimization}.
\bpublisher{Springer},
\blocation{Switzerland}
(\byear{2014}).
\doiurl{10.1007/978-3-319-08114-4}
\end{bbook}
\endbibitem

\bibitem[\protect\citeauthoryear{Burke et~al.}{2005}]{BLO2005}
\begin{barticle}
\bauthor{\bsnm{Burke}, \binits{J.V.}},
\bauthor{\bsnm{Lewis}, \binits{A.S.}},
\bauthor{\bsnm{Overton}, \binits{M.L.}}:
\batitle{{A Robust Gradient Sampling Algorithm for Nonsmooth, Nonconvex
  Optimization}}.
\bjtitle{{SIAM} Journal on Optimization}
\bvolume{15}(\bissue{3}),
\bfpage{751}--\blpage{779}
(\byear{2005})
\doiurl{10.1137/030601296}
\end{barticle}
\endbibitem

\bibitem[\protect\citeauthoryear{Burke et~al.}{2020}]{BCL2020}
\begin{bchapter}
\bauthor{\bsnm{Burke}, \binits{J.V.}},
\bauthor{\bsnm{Curtis}, \binits{F.E.}},
\bauthor{\bsnm{Lewis}, \binits{A.S.}},
\bauthor{\bsnm{Overton}, \binits{M.L.}},
\bauthor{\bsnm{Sim{\~{o}}es}, \binits{L.E.A.}}:
\bctitle{{Gradient Sampling Methods for Nonsmooth Optimization}}.
In: \bbtitle{Numerical Nonsmooth Optimization},
pp. \bfpage{201}--\blpage{225}.
\bpublisher{Springer},
\blocation{Switzerland}
(\byear{2020}).
\doiurl{10.1007/978-3-030-34910-3_6}
\end{bchapter}
\endbibitem

\bibitem[\protect\citeauthoryear{Mahdavi-Amiri and Yousefpour}{2012}]{MY2012}
\begin{barticle}
\bauthor{\bsnm{Mahdavi-Amiri}, \binits{N.}},
\bauthor{\bsnm{Yousefpour}, \binits{R.}}:
\batitle{{An Effective Nonsmooth Optimization Algorithm for Locally Lipschitz
  Functions}}.
\bjtitle{Journal of Optimization Theory and Applications}
\bvolume{155}(\bissue{1}),
\bfpage{180}--\blpage{195}
(\byear{2012})
\doiurl{10.1007/s10957-012-0024-7}
\end{barticle}
\endbibitem

\bibitem[\protect\citeauthoryear{Gebken and Peitz}{2021}]{GP2021}
\begin{barticle}
\bauthor{\bsnm{Gebken}, \binits{B.}},
\bauthor{\bsnm{Peitz}, \binits{S.}}:
\batitle{{An Efficient Descent Method for Locally Lipschitz Multiobjective
  Optimization Problems}}.
\bjtitle{Journal of Optimization Theory and Applications}
\bvolume{80},
\bfpage{3}--\blpage{29}
(\byear{2021})
\doiurl{10.1007/s10957-020-01803-w}
\end{barticle}
\endbibitem

\bibitem[\protect\citeauthoryear{Zhang et~al.}{2020}]{ZLJ2020}
\begin{bchapter}
\bauthor{\bsnm{Zhang}, \binits{J.}},
\bauthor{\bsnm{Lin}, \binits{H.}},
\bauthor{\bsnm{Jegelka}, \binits{S.}},
\bauthor{\bsnm{Sra}, \binits{S.}},
\bauthor{\bsnm{Jadbabaie}, \binits{A.}}:
\bctitle{Complexity of finding stationary points of nonconvex nonsmooth
  functions}.
In: \beditor{\bsnm{Daumé}, \binits{I.I.I.H.}},
\beditor{\bsnm{Singh}, \binits{A.}} (eds.)
\bbtitle{Proceedings of the 37th International Conference on Machine Learning}.
\bpublisher{PMLR},
\blocation{-}
(\byear{2020})
\end{bchapter}
\endbibitem

\bibitem[\protect\citeauthoryear{Davis et~al.}{2022}]{DDL2022}
\begin{bchapter}
\bauthor{\bsnm{Davis}, \binits{D.}},
\bauthor{\bsnm{Drusvyatskiy}, \binits{D.}},
\bauthor{\bsnm{Lee}, \binits{Y.T.}},
\bauthor{\bsnm{Padmanabhan}, \binits{S.}},
\bauthor{\bsnm{Ye}, \binits{G.}}:
\bctitle{A gradient sampling method with complexity guarantees for lipschitz
  functions in high and low dimensions}.
In: \beditor{\bsnm{Koyejo}, \binits{S.}},
\beditor{\bsnm{Mohamed}, \binits{S.}},
\beditor{\bsnm{Agarwal}, \binits{A.}},
\beditor{\bsnm{Belgrave}, \binits{D.}},
\beditor{\bsnm{Cho}, \binits{K.}},
\beditor{\bsnm{Oh}, \binits{A.}} (eds.)
\bbtitle{Advances in Neural Information Processing Systems},
vol. \bseriesno{35},
pp. \bfpage{6692}--\blpage{6703}.
\bpublisher{Curran Associates, Inc.},
\blocation{-}
(\byear{2022})
\end{bchapter}
\endbibitem

\bibitem[\protect\citeauthoryear{Goldstein}{1977}]{G1977}
\begin{barticle}
\bauthor{\bsnm{Goldstein}, \binits{A.A.}}:
\batitle{Optimization of lipschitz continuous functions}.
\bjtitle{Mathematical Programming}
\bvolume{13}(\bissue{1}),
\bfpage{14}--\blpage{22}
(\byear{1977})
\doiurl{10.1007/bf01584320}
\end{barticle}
\endbibitem

\bibitem[\protect\citeauthoryear{Helou et~al.}{2017}]{HSS2017}
\begin{barticle}
\bauthor{\bsnm{Helou}, \binits{E.S.}},
\bauthor{\bsnm{Santos}, \binits{S.A.}},
\bauthor{\bsnm{Sim{\~o}es}, \binits{L.E.A.}}:
\batitle{On the local convergence analysis of the gradient sampling method for
  finite max-functions}.
\bjtitle{J. Optim. Theory Appl.}
\bvolume{175}(\bissue{1}),
\bfpage{137}--\blpage{157}
(\byear{2017})
\end{barticle}
\endbibitem

\bibitem[\protect\citeauthoryear{Robinson}{1999}]{R1999}
\begin{barticle}
\bauthor{\bsnm{Robinson}, \binits{S.M.}}:
\batitle{Linear convergence of epsilon-subgradient descent methods for a class
  of convex functions}.
\bjtitle{Mathematical Programming}
\bvolume{86}(\bissue{1}),
\bfpage{41}--\blpage{50}
(\byear{1999})
\doiurl{10.1007/s101070050078}
\end{barticle}
\endbibitem

\bibitem[\protect\citeauthoryear{Díaz and Grimmer}{2023}]{DG2023}
\begin{barticle}
\bauthor{\bsnm{Díaz}, \binits{M.}},
\bauthor{\bsnm{Grimmer}, \binits{B.}}:
\batitle{{Optimal Convergence Rates for the Proximal Bundle Method}}.
\bjtitle{SIAM Journal on Optimization}
\bvolume{33}(\bissue{2}),
\bfpage{424}--\blpage{454}
(\byear{2023})
\doiurl{10.1137/21m1428601}
\end{barticle}
\endbibitem

\bibitem[\protect\citeauthoryear{Mifflin and Sagastizábal}{2012}]{MS2012}
\begin{botherref}
\oauthor{\bsnm{Mifflin}, \binits{R.}},
\oauthor{\bsnm{Sagastizábal}, \binits{C.}}:
{A science fiction story in nonsmooth optimization originating at IIASA}.
Documenta Mathematica
(2012)
\end{botherref}
\endbibitem

\bibitem[\protect\citeauthoryear{Hiriart-Urruty}{2007}]{H2007}
\begin{barticle}
\bauthor{\bsnm{Hiriart-Urruty}, \binits{J.-B.}}:
\batitle{{Potpourri of Conjectures and Open Questions in Nonlinear Analysis and
  Optimization}}.
\bjtitle{SIAM Review}
\bvolume{49},
\bfpage{255}--\blpage{273}
(\byear{2007})
\doiurl{10.1137/050633500}
\end{barticle}
\endbibitem

\bibitem[\protect\citeauthoryear{Mifflin and Sagastizábal}{2005}]{MS2005}
\begin{barticle}
\bauthor{\bsnm{Mifflin}, \binits{R.}},
\bauthor{\bsnm{Sagastizábal}, \binits{C.}}:
\batitle{{A VU-algorithm for convex minimization}}.
\bjtitle{Mathematical Programming}
\bvolume{104},
\bfpage{583}--\blpage{608}
(\byear{2005})
\doiurl{10.1007/s10107-005-0630-3}
\end{barticle}
\endbibitem

\bibitem[\protect\citeauthoryear{Daniilidis et~al.}{2009}]{DSS2009}
\begin{barticle}
\bauthor{\bsnm{Daniilidis}, \binits{A.}},
\bauthor{\bsnm{Sagastizábal}, \binits{C.}},
\bauthor{\bsnm{Solodov}, \binits{M.}}:
\batitle{Identifying structure of nonsmooth convex functions by the bundle
  technique}.
\bjtitle{SIAM Journal on Optimization}
\bvolume{20}(\bissue{2}),
\bfpage{820}--\blpage{840}
(\byear{2009})
\doiurl{10.1137/080729864}
\end{barticle}
\endbibitem

\bibitem[\protect\citeauthoryear{Charisopoulos and Davis}{2024}]{CD2024}
\begin{barticle}
\bauthor{\bsnm{Charisopoulos}, \binits{V.}},
\bauthor{\bsnm{Davis}, \binits{D.}}:
\batitle{{A Superlinearly Convergent Subgradient Method for Sharp Semismooth
  Problems}}.
\bjtitle{Mathematics of Operations Research}
\bvolume{49},
\bfpage{1678}--\blpage{1709}
(\byear{2024})
\doiurl{10.1287/moor.2023.1390}
\end{barticle}
\endbibitem

\bibitem[\protect\citeauthoryear{Polyak}{1969}]{P1969}
\begin{barticle}
\bauthor{\bsnm{Polyak}, \binits{B.T.}}:
\batitle{Minimization of unsmooth functionals}.
\bjtitle{USSR Computational Mathematics and Mathematical Physics}
\bvolume{9}(\bissue{3}),
\bfpage{14}--\blpage{29}
(\byear{1969})
\doiurl{10.1016/0041-5553(69)90061-5}
\end{barticle}
\endbibitem

\bibitem[\protect\citeauthoryear{Mifflin}{1984}]{M1984}
\begin{bchapter}
\bauthor{\bsnm{Mifflin}, \binits{R.}}:
\bctitle{Better than linear convergence and safeguarding in nonsmooth
  minimization}.
In: \bbtitle{System Modelling and Optimization},
pp. \bfpage{321}--\blpage{330}.
\bpublisher{Springer},
\blocation{Berlin, Heidelberg}
(\byear{1984})
\end{bchapter}
\endbibitem

\bibitem[\protect\citeauthoryear{Luk{\v s}an and Vl{\v c}ek}{1998}]{LV1998}
\begin{barticle}
\bauthor{\bsnm{Luk{\v s}an}, \binits{L.}},
\bauthor{\bsnm{Vl{\v c}ek}, \binits{J.}}:
\batitle{{A bundle-Newton method for nonsmooth unconstrained minimization}}.
\bjtitle{Mathematical Programming}
\bvolume{83}(\bissue{1-3}),
\bfpage{373}--\blpage{391}
(\byear{1998})
\end{barticle}
\endbibitem

\bibitem[\protect\citeauthoryear{Grothey}{2002}]{G2002}
\begin{botherref}
\oauthor{\bsnm{Grothey}, \binits{A.}}:
{A Second Order Trust Region Bundle Method for Nonconvex Nonsmooth
  Optimization}.
Technical report MS-02-005, University of Edinburgh, Edinburgh
(2002)
\end{botherref}
\endbibitem

\bibitem[\protect\citeauthoryear{Curtis and Overton}{2012}]{CO2012}
\begin{barticle}
\bauthor{\bsnm{Curtis}, \binits{F.E.}},
\bauthor{\bsnm{Overton}, \binits{M.L.}}:
\batitle{{A Sequential Quadratic Programming Algorithm for Nonconvex, Nonsmooth
  Constrained Optimization}}.
\bjtitle{{SIAM} Journal on Optimization}
\bvolume{22}(\bissue{2}),
\bfpage{474}--\blpage{500}
(\byear{2012})
\doiurl{10.1137/090780201}
\end{barticle}
\endbibitem

\bibitem[\protect\citeauthoryear{Curtis and Que}{2015}]{CQ2015}
\begin{barticle}
\bauthor{\bsnm{Curtis}, \binits{F.E.}},
\bauthor{\bsnm{Que}, \binits{X.}}:
\batitle{{A quasi-Newton algorithm for nonconvex, nonsmooth optimization with
  global convergence guarantees}}.
\bjtitle{Mathematical Programming Computation}
\bvolume{7}(\bissue{4}),
\bfpage{399}--\blpage{428}
(\byear{2015})
\doiurl{10.1007/s12532-015-0086-2}
\end{barticle}
\endbibitem

\bibitem[\protect\citeauthoryear{Curtis et~al.}{2019}]{CRZ2019}
\begin{barticle}
\bauthor{\bsnm{Curtis}, \binits{F.E.}},
\bauthor{\bsnm{Robinson}, \binits{D.P.}},
\bauthor{\bsnm{Zhou}, \binits{B.}}:
\batitle{A self-correcting variable-metric algorithm framework for nonsmooth
  optimization}.
\bjtitle{IMA Journal of Numerical Analysis}
\bvolume{40},
\bfpage{1154}--\blpage{1187}
(\byear{2019})
\doiurl{10.1093/imanum/drz008}
\end{barticle}
\endbibitem

\bibitem[\protect\citeauthoryear{Lukšan and Vlček}{1999}]{LV1999}
\begin{barticle}
\bauthor{\bsnm{Lukšan}, \binits{L.}},
\bauthor{\bsnm{Vlček}, \binits{J.}}:
\batitle{{Globally Convergent Variable Metric Method for Convex Nonsmooth
  Unconstrained Minimization}}.
\bjtitle{Journal of Optimization Theory and Applications}
\bvolume{102},
\bfpage{593}--\blpage{613}
(\byear{1999})
\doiurl{10.1023/a:1022650107080}
\end{barticle}
\endbibitem

\bibitem[\protect\citeauthoryear{Haarala}{2004}]{H2004}
\begin{botherref}
\oauthor{\bsnm{Haarala}, \binits{M.}}:
{Large-Scale Nonsmooth Optimization: Variable Metric Bundle Method with Limited
  Memory}.
PhD thesis,
University of Jyväskylä
(2004)
\end{botherref}
\endbibitem

\bibitem[\protect\citeauthoryear{Lewis and Overton}{2013}]{LO2013}
\begin{barticle}
\bauthor{\bsnm{Lewis}, \binits{A.S.}},
\bauthor{\bsnm{Overton}, \binits{M.L.}}:
\batitle{{Nonsmooth optimization via quasi-Newton methods}}.
\bjtitle{Mathematical Programming}
\bvolume{141},
\bfpage{135}--\blpage{163}
(\byear{2013})
\doiurl{10.1007/s10107-012-0514-2}
\end{barticle}
\endbibitem

\bibitem[\protect\citeauthoryear{Curtis et~al.}{2017}]{CMO2017}
\begin{barticle}
\bauthor{\bsnm{Curtis}, \binits{F.E.}},
\bauthor{\bsnm{Mitchell}, \binits{T.}},
\bauthor{\bsnm{Overton}, \binits{M.L.}}:
\batitle{A {BFGS}-{SQP} method for nonsmooth, nonconvex, constrained
  optimization and its evaluation using relative minimization profiles}.
\bjtitle{Optimization Methods and Software}
\bvolume{32}(\bissue{1}),
\bfpage{148}--\blpage{181}
(\byear{2017})
\doiurl{10.1080/10556788.2016.1208749}
\end{barticle}
\endbibitem

\bibitem[\protect\citeauthoryear{Goodfellow et~al.}{2016}]{GBC2016}
\begin{bbook}
\bauthor{\bsnm{Goodfellow}, \binits{I.}},
\bauthor{\bsnm{Bengio}, \binits{Y.}},
\bauthor{\bsnm{Courville}, \binits{A.}}:
\bbtitle{Deep Learning}.
\bpublisher{MIT Press},
\blocation{-}
(\byear{2016}).
\bcomment{\url{http://www.deeplearningbook.org}}
\end{bbook}
\endbibitem

\bibitem[\protect\citeauthoryear{Boyd and Vandenberghe}{2004}]{BV2004}
\begin{bbook}
\bauthor{\bsnm{Boyd}, \binits{S.}},
\bauthor{\bsnm{Vandenberghe}, \binits{L.}}:
\bbtitle{Convex Optimization}.
\bpublisher{Cambridge University Press},
\blocation{Cambridge}
(\byear{2004}).
\doiurl{10.1017/cbo9780511804441}
\end{bbook}
\endbibitem

\bibitem[\protect\citeauthoryear{Clarke}{1990}]{C1990}
\begin{bbook}
\bauthor{\bsnm{Clarke}, \binits{F.H.}}:
\bbtitle{Optimization and Nonsmooth Analysis}.
\bpublisher{Society for Industrial and Applied Mathematics},
\blocation{Philadelphia}
(\byear{1990}).
\doiurl{10.1137/1.9781611971309}
\end{bbook}
\endbibitem

\bibitem[\protect\citeauthoryear{Lemaréchal}{1989}]{L1989}
\begin{bbook}
\bauthor{\bsnm{Lemaréchal}, \binits{C.}}:
\bbtitle{Chapter VII. Nondifferentiable Optimization},
pp. \bfpage{529}--\blpage{572}.
\bpublisher{Elsevier},
\blocation{-}
(\byear{1989}).
\doiurl{10.1016/s0927-0507(89)01008-x}
\end{bbook}
\endbibitem

\bibitem[\protect\citeauthoryear{Kiwiel}{2007}]{K2007}
\begin{barticle}
\bauthor{\bsnm{Kiwiel}, \binits{K.C.}}:
\batitle{{Convergence of the Gradient Sampling Algorithm for Nonsmooth
  Nonconvex Optimization}}.
\bjtitle{{SIAM} Journal on Optimization}
\bvolume{18}(\bissue{2}),
\bfpage{379}--\blpage{388}
(\byear{2007})
\doiurl{10.1137/050639673}
\end{barticle}
\endbibitem

\bibitem[\protect\citeauthoryear{Cheney and Goldstein}{1959}]{CG1959}
\begin{barticle}
\bauthor{\bsnm{Cheney}, \binits{W.}},
\bauthor{\bsnm{Goldstein}, \binits{A.A.}}:
\batitle{Proximity maps for convex sets}.
\bjtitle{Proceedings of the American Mathematical Society}
\bvolume{10}(\bissue{3}),
\bfpage{448}--\blpage{448}
(\byear{1959})
\doiurl{10.1090/s0002-9939-1959-0105008-8}
\end{barticle}
\endbibitem

\bibitem[\protect\citeauthoryear{Golubitsky and Guillemin}{1973}]{GG1973}
\begin{bbook}
\bauthor{\bsnm{Golubitsky}, \binits{M.}},
\bauthor{\bsnm{Guillemin}, \binits{V.}}:
\bbtitle{Stable Mappings and Their Singularities}.
\bpublisher{Springer},
\blocation{New York}
(\byear{1973}).
\doiurl{10.1007/978-1-4615-7904-5}
\end{bbook}
\endbibitem

\bibitem[\protect\citeauthoryear{Ioffe and Penot}{1997}]{IP1997}
\begin{barticle}
\bauthor{\bsnm{Ioffe}, \binits{A.}},
\bauthor{\bsnm{Penot}, \binits{J.-P.}}:
\batitle{Limiting subhessians, limiting subjets and their calculus}.
\bjtitle{Transactions of the American Mathematical Society}
\bvolume{349}(\bissue{2}),
\bfpage{789}--\blpage{807}
(\byear{1997})
\doiurl{10.1090/s0002-9947-97-01726-1}
\end{barticle}
\endbibitem

\bibitem[\protect\citeauthoryear{Horvath}{1985}]{H1985}
\begin{barticle}
\bauthor{\bsnm{Horvath}, \binits{C.}}:
\batitle{Measure of non-compactness and multivalued mappings in complete metric
  topological vector spaces}.
\bjtitle{Journal of Mathematical Analysis and Applications}
\bvolume{108}(\bissue{2}),
\bfpage{403}--\blpage{408}
(\byear{1985})
\doiurl{10.1016/0022-247x(85)90033-2}
\end{barticle}
\endbibitem

\bibitem[\protect\citeauthoryear{Scholtes}{2012}]{S2012}
\begin{bbook}
\bauthor{\bsnm{Scholtes}, \binits{S.}}:
\bbtitle{Introduction to Piecewise Differentiable Equations}.
\bpublisher{Springer},
\blocation{New York, NY}
(\byear{2012}).
\doiurl{10.1007/978-1-4614-4340-7}
\end{bbook}
\endbibitem

\bibitem[\protect\citeauthoryear{Rockafellar and Wets}{1998}]{RW1998}
\begin{bbook}
\bauthor{\bsnm{Rockafellar}, \binits{R.T.}},
\bauthor{\bsnm{Wets}, \binits{R.J.B.}}:
\bbtitle{Variational Analysis}.
\bpublisher{Springer},
\blocation{Berlin, Heidelberg}
(\byear{1998}).
\doiurl{10.1007/978-3-642-02431-3}
\end{bbook}
\endbibitem

\bibitem[\protect\citeauthoryear{K{\"{o}}nigsberger}{2004}]{K2004}
\begin{bbook}
\bauthor{\bsnm{K{\"{o}}nigsberger}, \binits{K.}}:
\bbtitle{Analysis 2}.
\bpublisher{Springer},
\blocation{Berlin, Heidelberg}
(\byear{2004}).
\doiurl{10.1007/3-540-35077-2}
\end{bbook}
\endbibitem

\bibitem[\protect\citeauthoryear{Noll}{2010}]{N2010}
\begin{barticle}
\bauthor{\bsnm{Noll}, \binits{D.}}:
\batitle{{Cutting Plane Oracles to Minimize Non-smooth Non-convex Functions}}.
\bjtitle{Set-Valued and Variational Analysis}
\bvolume{18}(\bissue{3–4}),
\bfpage{531}--\blpage{568}
(\byear{2010})
\doiurl{10.1007/s11228-010-0159-3}
\end{barticle}
\endbibitem

\bibitem[\protect\citeauthoryear{Conn et~al.}{2000}]{CGT2000}
\begin{bbook}
\bauthor{\bsnm{Conn}, \binits{A.R.}},
\bauthor{\bsnm{Gould}, \binits{N.I.M.}},
\bauthor{\bsnm{Toint}, \binits{P.L.}}:
\bbtitle{Trust Region Methods}.
\bpublisher{Society for Industrial and Applied Mathematics},
\blocation{Philadelphia}
(\byear{2000}).
\doiurl{10.1137/1.9780898719857}
\end{bbook}
\endbibitem

\bibitem[\protect\citeauthoryear{Nocedal and Wright}{2006}]{NW2006}
\begin{bbook}
\bauthor{\bsnm{Nocedal}, \binits{J.}},
\bauthor{\bsnm{Wright}, \binits{S.}}:
\bbtitle{Numerical Optimization}.
\bpublisher{Springer},
\blocation{New York, NY}
(\byear{2006}).
\doiurl{10.1007/978-0-387-40065-5}
\end{bbook}
\endbibitem

\bibitem[\protect\citeauthoryear{Wächter and Biegler}{2005}]{WB2005}
\begin{barticle}
\bauthor{\bsnm{Wächter}, \binits{A.}},
\bauthor{\bsnm{Biegler}, \binits{L.T.}}:
\batitle{On the implementation of an interior-point filter line-search
  algorithm for large-scale nonlinear programming}.
\bjtitle{Mathematical Programming}
\bvolume{106}(\bissue{1}),
\bfpage{25}--\blpage{57}
(\byear{2005})
\doiurl{10.1007/s10107-004-0559-y}
\end{barticle}
\endbibitem

\bibitem[\protect\citeauthoryear{Bertolazzi}{2023}]{B2023}
\begin{botherref}
\oauthor{\bsnm{Bertolazzi}, \binits{E.}}:
{ebertolazzi/mexIPOPT}.
GitHub (Retrieved November 25, 2024)
(2023)
\end{botherref}
\endbibitem

\bibitem[\protect\citeauthoryear{Gebken}{2024a}]{G2024b}
\begin{barticle}
\bauthor{\bsnm{Gebken}, \binits{B.}}:
\batitle{A note on the convergence of deterministic gradient sampling in
  nonsmooth optimization}.
\bjtitle{Computational Optimization and Applications}
\bvolume{88},
\bfpage{151}--\blpage{165}
(\byear{2024})
\doiurl{10.1007/s10589-024-00552-0}
\end{barticle}
\endbibitem

\bibitem[\protect\citeauthoryear{Gebken}{2024b}]{G2024a}
\begin{botherref}
\oauthor{\bsnm{Gebken}, \binits{B.}}:
Analyzing the speed of convergence in nonsmooth optimization via the Goldstein
  subdifferential with application to descent methods.
arXiv
(2024).
\doiurl{10.48550/arXiv.2410.01382}
\end{botherref}
\endbibitem

\bibitem[\protect\citeauthoryear{Haarala et~al.}{2004}]{HMM2004}
\begin{barticle}
\bauthor{\bsnm{Haarala}, \binits{M.}},
\bauthor{\bsnm{Miettinen}, \binits{K.}},
\bauthor{\bsnm{M\"{a}kel\"{a}}, \binits{M.M.}}:
\batitle{New limited memory bundle method for large-scale nonsmooth
  optimization}.
\bjtitle{Optimization Methods and Software}
\bvolume{19}(\bissue{6}),
\bfpage{673}--\blpage{692}
(\byear{2004})
\doiurl{10.1080/10556780410001689225}
\end{barticle}
\endbibitem

\bibitem[\protect\citeauthoryear{Lewis and Overton}{2008}]{LO2008}
\begin{botherref}
\oauthor{\bsnm{Lewis}, \binits{A.}},
\oauthor{\bsnm{Overton}, \binits{M.L.}}:
{Nonsmooth Optimization via BFGS}
(2008).
\url{https://optimization-online.org/?p=10625}
\end{botherref}
\endbibitem

\bibitem[\protect\citeauthoryear{Linderoth}{2005}]{L2005}
\begin{barticle}
\bauthor{\bsnm{Linderoth}, \binits{J.}}:
\batitle{A simplicial branch-and-bound algorithm for solving quadratically
  constrained quadratic programs}.
\bjtitle{Mathematical Programming}
\bvolume{103}(\bissue{2}),
\bfpage{251}--\blpage{282}
(\byear{2005})
\doiurl{10.1007/s10107-005-0582-7}
\end{barticle}
\endbibitem

\bibitem[\protect\citeauthoryear{Niculescu and Persson}{2006}]{NP2006}
\begin{bbook}
\bauthor{\bsnm{Niculescu}, \binits{C.P.}},
\bauthor{\bsnm{Persson}, \binits{L.-E.}}:
\bbtitle{Convex Functions and Their Applications}.
\bpublisher{Springer},
\blocation{New York}
(\byear{2006}).
\doiurl{10.1007/0-387-31077-0}
\end{bbook}
\endbibitem

\end{thebibliography}

\end{document}